\documentclass[12pt,reqno]{amsart}
\usepackage{amssymb,amsmath,bm,bbm}
\usepackage{amsfonts,graphicx,amsthm,amstext,latexsym,amsfonts}
\usepackage{graphicx}
\usepackage{caption}
\usepackage{color}
\usepackage[usenames,dvipsnames]{xcolor}
\usepackage{tikz}
\usepackage{hyperref}
\usepackage[colorinlistoftodos]{todonotes}
\usepackage{bbm}
\newcommand{\bb}{{\mathbbm{1}}}

\def\eps{\varepsilon}
\def\d{{\rm d}}

\def\R {\mathbb{R}}

\def\O{\Omega}
\newtheorem{proposition}{Proposition}[section]
\newtheorem{theorem}[proposition]{Theorem}
\newtheorem{corollary}[proposition]{Corollary}
\newtheorem{lemma}[proposition]{Lemma}
\theoremstyle{definition}

\numberwithin{equation}{section}
\newcommand{\UUU}{\color{black}}

\newcommand{\EEE}{\color{black}}
\newcommand{\MMM}{\color{black}}

\newcommand{\AAA}{\color{black}}
\newcommand{\AAB}{\color{black}}
\newcommand{\BBB}{\color{black}}

\newcommand{\tr}{\mathrm{tr}\,}
\newcommand{\cof}{\mathrm{cof\,}}

\newcommand{\MK}{\color{black}}

\newcommand{\EM}{\color{black}}

\newcommand{\KKK}{\color{black}}
\newcommand{\EOR}{\color{black}}

\newcommand{\n}{\mathbf n}
\newcommand{\om}{\Omega}

\newcommand{\dv}{\mathrm{div}\,}

\newenvironment{proofad}{\removelastskip\par\medskip   
\noindent{\bf Proof of { {\bf Theorem \ref{mainth1}}}.}
\rm}{\penalty-20\null\hfill$\square$\par\medbreak} 

\newenvironment{proofad2}{\removelastskip\par\medskip   
\noindent{\bf Proof of { {\bf Theorem \ref{main2}}}.}
\rm}{\penalty-20\null\hfill$\square$\par\medbreak} 

\newenvironment{proofad4}{\removelastskip\par\medskip   
\noindent{\bf Proof of { {\bf Theorem \ref{main3}}}.}
\rm}{\penalty-20\null\hfill$\square$\par\medbreak} 

\newenvironment{proofad3}{\removelastskip\par\medskip   
\noindent{\bf Proof of { {\bf Corollary \ref{fandi}}}.}
\rm}{\penalty-20\null\hfill$\square$\par\medbreak} 

\def \no#1#2#3 {{\bf #1} (#3), #2.}
\def \eds#1#2#3 {#1, #2, #3.}
\title[Linearization of finite elasticity with surface tension]
{Linearization of finite elasticity with surface tension}

\author{Martin Kru\v{z}\'ik}
\address[Martin Kru\v{z}\'ik]{
Czech Academy of Sciences, Institute of Information Theory and Automation,
Pod vod\' arenskou ve\v z\' \i\ 4, 182 08, Prague 8, Czech Republic and Faculty of Civil Engineering, Czech Technical University, Th\'{a}kurova 7, 166 29, Prague 6, Czech Republic.}
\email{kruzik@utia.cas.cz}

\author{Edoardo Mainini}
\address[Edoardo Mainini]{Dipartimento di Ingegneria meccanica, energetica, gestionale e dei trasporti, 
  Universit\`a  degli studi di Genova, Via all'Opera Pia, 15 - 16145 Genova Italy.}
\email{edoardo.mainini@unige.it}

\begin{document}
\begin{abstract}
\MK 
We propose    models in nonlinear elasticity for nonsimple materials that include  surface energy terms. Additionally, we also discuss  living surface loads   on the boundary.  \EM We establish corresponding linearized models  and show their relationship to the original ones  by means of  $\Gamma$-convergence. 
\EOR
\end{abstract}
\maketitle
%
%
%
\section{Introduction}\label{intro}

Surfaces of elastic  bodies  exhibit properties that are different from those associated with the
bulk. This behavior is caused either by the fact that the boundary
of the material is exposed to fatigue, chemical processes, coating, etc., thus
obviously resulting in very  different properties in comparatively thin boundary layers, or due to the fact that the atomic bonds
are broken at the surface of the body. These effects can be  phenomenologically  modeled in terms of boundaries equipped with
their own stored  energy density and they  have been well studied in the literature. Rational continuum mechanics approach to elastic surfaces 
has started with the work of Gurtin and Murdoch \cite{Gurtin75, G} and of  Steigmann and Ogden \cite{Steigmann-Ogden}. \AAB General energies for modeling body-environment interactions and surface loading are discussed in \cite{PG,PGVC}. \EEE  More recently, 
\v{S}ilhav\'{y} introduced a concept of interfacial quasiconvexity and polyconvexity \cite{S1, S2, S3} extending  \cite{F} to establish the existence of minimizers for problems with multi-well bulk energy density for simple materials, i.e., depending only on the first deformation gradient. Javili and Steinmann \cite{Javili-Steinmann-I, Javili-Steinmann-II} designed finite-element methods to model surface elasticity. Thermomechanical approach to interfacial and surface energetics of materials is reviewed in \cite{Javili, Steinmann-Haesner}. Surface-substrate interactions for shells is described in \cite{S4} and a phase-field modeling approach to the problem can be found e.g.~in \cite{Levitas14b}.

Rigorous derivation  of linear models from nonlinear elasticity by $\Gamma$-convergence \cite{DalMaso93} has  started by a pioneering work \cite{DMNP} and then various finer results appeared \cite{ADMDS,MPTJOTA,MPTARMA,MP2,MPsharp, MP, MOP, MM} 
that derive linear elasticity under various constraints as, e.g., incompressibility, or no Dirichlet boundary conditions, i.e., pure traction problems.  We also  refer to \cite{mora-riva} where the authors performed the linearization procedure for a pressure load, i.e., living load (a follower force) that depends on the deformation.  
\AAB
The study \cite{Dias}  explores the macroscopic elastic properties of elastomers with spherical cavities filled with pressurized liquid, considering the effects of surface tension. It starts by linearizing a fully nonlinear model and progresses with analyzing how the presence of multiple liquid-filled cavities enhances the elastic behavior in the linearized model. \EEE

 Our contribution focuses on the linearization of nonlinear elasticity models  with surface energy terms. \MK In order to fix ideas, we perform the analysis on the model functional $\mathcal{G}$ below  that  penalizes local changes of the surface area in the deformed configuration in analogy with the  incompressibility constraint in the bulk. \EOR
The main results are given in Theorem \ref{mainth1}  and Theorem \ref{main2},  stating the convergence of  minimizers of nonlinear models to the unique minimizer of the associated  linearized models, in case of Dirichlet boundary conditions and in case of only Neumann boundary conditions, respectively.
In the remainder of this section, we introduce the necessary notation and the functionals. The main results are then stated in Section~\ref{Sec:Results} and Sections~\ref{Sec:Proof1} and \ref{Sec:Proof2} are devoted to proofs.

\subsection*{\AAB Modeling of elastic surface energy\BBB}
Let $\Omega\subset \mathbb R^d$, $d\ge 2$, be a bounded open connected Lipschitz set, representing the reference configuration of a hyperelastic  body.
Let us introduce the energy functional depending on the deformation field $y:\Omega\to\mathbb R^d$
\[
\mathcal G(y):=\int_\Omega \left(W(\nabla y(x))+H(\nabla^2 y(x))\right)dx -\overline{\mathcal L}(y(x)-x)+\gamma\AAA
\left\|\,|\cof\nabla y\, \n|-1\,\right\|_{L^q(\partial\Omega)}^q.\BBB
\]
The global energy includes the stored elastic energy of a nonsimple material, as it depends on second gradient \cite{Toupin:62, Toupin:64}.  Here, $\overline{\mathcal L}$ is a dead loading functional, accounting for the work of given external force fields. It is a linear functional acting on the displacement field $y(x)-x$. 
Moreover, local changes in the surface measure of the boundary of $\Omega$ are penalized, as $|\cof\nabla y\,\mathbf n|$ represents the density of the surface area element of the deformed configuration:  the energy includes the \AAA $L^q(\partial\Omega)$ distance, $q\ge 1$, \BBB to its reference value corresponding to $y(x)=x$, being $\gamma>0$ a surface tension coefficient.

In the linearization process, we view $y$ as a perturbation of the identity map so that we write it as $y(x)=x+\eps v(x)$ for a suitable rescaled displacement $v$ \MK and we set $\mathcal{L}=\overline{\mathcal{L}}/\varepsilon$, which reflects the scaling of external forces.      Assuming \EM that $H$ and $W$ are frame indifferent, that  $H$ is convex positively $p$-homogeneous for some suitable $p>1$ and that $W$ is minimized at the identity with   \EOR $W(\mathbf I)=D W(\mathbf I)=0$, i.e., that  identity is the natural state of the body $\Omega$,   we consider the rescaled nonlinear global energy \EOR
\begin{equation}\begin{aligned}\label{Geps}
\mathcal G_\eps(v)&:=\frac1{\eps^2}\int_\Omega W(\mathbf I+\eps\nabla v)\,dx+ \MK \frac1{\eps^p}\int_\Omega H(\varepsilon\nabla^2 v) -\mathcal L(v)\EOR\\&\qquad\qquad\qquad+\AAA\frac\gamma{\eps^q}\,\|\,|\cof(\mathbf I+\eps \nabla v)\, \n|-1\,\|_{L^q(\partial \Omega)}^q\BBB.
\end{aligned}\end{equation}
We notice that as $\eps\to 0$  $$W(\mathbf I+\eps\nabla v)=\frac{\eps^2}2\,\mathbb  E(v)\,\nabla^2W(\mathbf I)\,\mathbb E(v) +o(\eps^2),$$
so that the second order term in the Taylor expansion of $W$ produces the standard quadratic potential of linear elasticity (being $\nabla^2W(\mathbf I)$ the fourth order elastic tensor), acting by frame indifference  only on the symmetric part of the gradient $$\mathbb E(v):=\frac{\nabla v+(\nabla v)^T}2.$$
On the other hand, by the formula $\cof \mathbf F=(\det \mathbf F)\, \mathbf F^{-T}$ we have
\begin{equation}\label{cofexpansion}\begin{aligned}
\cof(\mathbf I+\eps\nabla v)&=\det(\mathbf I+\eps\nabla v)\,(\mathbf I+\eps(\nabla v)^T)^{-1}\\&=(1+\eps\,\dv v+o(\eps))\,(\mathbf I-\eps(\nabla v)^T+o(\eps))=\mathbf I+\eps\mathbb A(v)+o(\eps),
\end{aligned}\end{equation}
where we have introduced the divergence free tensor $$\mathbb A(v)=\mathbf I\,\dv v-(\nabla v)^T,$$ corresponding to the linearization of the cofactor matrix. As a consequence,
\begin{equation}\label{rootexpansion}
|\cof(\mathbf I+\eps\nabla v)\n|=\sqrt{1+2\eps\,\mathbb A(v)\mathbf n\cdot\n+o(\eps)}=1+\eps\,\mathbb A(v)\n\cdot \n+o(\eps),
\end{equation}
\AAA where we notice that $\mathbb A(v)\n\cdot\n$ is the tangential divergence of $v$ on $\partial\Omega$.  \BBB
Therefore in the limit as $\eps\to 0$ we obtain the geometrically linearized functional
\begin{equation}\label{G*}
\mathcal G_*(v):=\frac12\int_\Omega \mathbb E(v)\,D^2 W(\mathbf I)\,\mathbb E(v)\,dx+\int_\Omega H(\nabla^2 v)\,dx -\mathcal L(v)+\gamma \AAA\int_{\partial \Omega}|\mathbb A(v)\n\cdot\n|^q\,dS\BBB
\end{equation}
as the pointwise limit of functionals $\mathcal G_h$ for every smooth enough $v$. We notice that the stored elastic energy and the surface tension term of functional $\mathcal G$ are frame indifferent, and indeed their counterparts in functional $\mathcal G_*$ depend on $\nabla v$ only through its symmetric part $\mathbb E(v)$, so that as expected they are  invariant by addition of an infinitesimal rigid displacement.  We have indeed $\nabla^2(c+\mathbf Wx)=0$ and $\mathbb A(c+\mathbf Wx)\n\cdot\n=\mathbf W\n\cdot\n=0$ for every $c\in\mathbb R^d$ and every $\mathbf W \in\mathbb R_{\mathrm{skew}}^{d\times d}$. \EM We stress that the natural choice $H(\nabla^2 v)=|\nabla^2v|
^2$ \AAA and $q=2$, yielding a quadratic stored elastic energy in \eqref{G*} and a quadratic surface energy term, \BBB is possible within our theory in the physical cases $d=2,3$, as our main results will be given under the \AAA restriction $p\ge dq/(q+1)$. \EOR


\MMM\subsection*{Alternative examples of surface  energies}
Let us next introduce possible alternative models, for which we will prove the validity of the same convergence results. \MK We can consider a surface energy term that penalizes differences in \AAB total \EEE surface areas between the reference and deformed configurations, e.g.,  
\[
\mathcal F(y):= \int_\Omega (W(\nabla y)+ H(\nabla^2 y))\,dx-\overline{\mathcal L}(y(x)-x)+\gamma \AAA\left|\int_{\partial\Omega}|\cof\nabla y(\sigma)\,\n(\sigma)|\,dS(\sigma) -|\partial\Omega|\right|^q\BBB.
\]
The change-of-variables formula for  surface integrals indicates that the last term is equal to \AAA  $\gamma\,||\partial\Omega^y|-|\partial\Omega||^q$, \MMM where $\partial\Omega^y$ denotes the boundary of the deformed configuration $y(\Omega)$\MK. Here $\gamma>0$ is again a physical parameter.
The associated rescaled energies are given by
\begin{equation}\label{Feps}\begin{aligned}
\mathcal F_\eps(v):=& \frac1{\eps^2}\int_\Omega W(\mathbf I+\eps \nabla v)\,dx+\MK \frac1{\eps^p}\int_\Omega H(\varepsilon\nabla^2v)\EOR-\mathcal L(v)\\&\qquad\qquad\qquad+\AAA \frac\gamma{\eps^q} \left|\int_{\partial\Omega}|\cof(\mathbf I+\eps\nabla v)\,\mathbf n|\,dS -|\partial\Omega|\right|^q\BBB.
\end{aligned}\end{equation}
Taking into account \eqref{cofexpansion} and \eqref{rootexpansion} we see that the as $\eps\to 0$
\[
\left|\int_{\partial\Omega}|\cof(\mathbf I+\eps\nabla v)\,\n|\,dS -|\partial\Omega|\right|=\eps\left|\int_{\partial\Omega}\mathbb A(v)\n\cdot\n\,dS\right|+o(\eps),
\] 
\MMM
so that  the formal linearized  functional obtained as $\eps\to 0$ is
 \begin{equation}\label{F*}
 \mathcal F_*(v):= \frac1{2}\int_\Omega \mathbb E(v)\,\nabla^2W(\mathbf I)\,\mathbb E(v)\,dx+\int_\Omega H(\nabla^2 v)\,dx-\mathcal L(v)+\gamma\AAA \left|\int_{\partial\Omega} \mathbb A(v)\,\n\cdot\n\,dS\right|^q\BBB.
 \end{equation}
We notice that \MK  the divergence theorem implies, assuming enough smoothness of $\partial\Omega$ and introducing a suitable extension of the normal $\n$ to the whole of $\Omega$,
\[\int_{\partial\Omega}\mathbb A(v)\,\n\cdot\n\,dS=\int_{\partial\Omega}\mathbb A(v)^T\n\cdot\n\,dS=
\int_{\Omega}\dv(\mathbb A(v)^T\n)\,dx=\int_{\Omega}\mathbb A(v): \nabla \n\,dx,
 \]
 having used the fact that $\mathbb A(v)$ is divergence free. Therefore \MMM functional $\mathcal F_*$ can also be rewritten as \MK
 \begin{equation*}
 \mathcal F_*(v)= \frac1{2}\int_\Omega \mathbb E(v)\,\nabla^2W(\mathbf I)\,\mathbb E(v)\,dx+\int_\Omega H(\nabla^2 v)\,dx-\mathcal L(v)+\AAA \gamma \left|\int_\Omega \mathbb A(v):\nabla\n\,dx\right|^q\BBB.
 \end{equation*}
\MK 

\EOR
 

 We can also consider \MMM a model featuring \KKK surface loading 
\[
\mathcal I(y):= \int_\Omega (W(\nabla y)+H(\nabla^2 y))\,dx-\overline{\mathcal L}(y(x)-x)+\AAA\gamma\int_{\partial\Omega}|(\cof\nabla y(\sigma)-\mathbf I)\,\n(\sigma)|^q\,dS(\sigma)\BBB.
\]
The last term \AAA features the $L^q(\partial\Omega;\mathbb R^3)$ \MMM distance between the normal vector $\n^y=\cof\nabla y\,\n$ in the deformed configuration and the undeformed one corresponding to $y(x)=x$. This restricts energetically favorable deformations in a stronger way as it penalizes local changes not only in the area of the boundary but  also in the orientation of the boundary, and indeed the last term pays energy also for rigid motions of the reference configuration. Again we shall introduce the rescaled energies
\begin{equation*}\begin{aligned}
\mathcal I_\eps(v):&= \frac1{\eps^2}\int_\Omega W(\mathbf I+\eps \nabla v)\,dx+\MK \frac1{\varepsilon^p}\int_\Omega H(\varepsilon\nabla^2 v)\EOR-\mathcal L(v)\\&\qquad\qquad\qquad+\AAA\frac\gamma{\eps^q} \int_{\partial\Omega}\left|(\cof(\mathbf I+\eps\nabla v)-\mathbf I)\,\n\right|^q\,dS.\AAA
\end{aligned}\end{equation*}
Thanks to \eqref{cofexpansion} we see that the formal limiting functional as $\eps\to 0$ is 
 \begin{equation*}
 \mathcal I_*(v)= \frac1{2}\int_\Omega \mathbb E(v)\,\nabla^2W(\mathbf I)\,\mathbb E(v)\,dx+\int_\Omega H(\nabla^2 v)\,dx-\mathcal L(v)+\AAA\gamma \int_{\partial\Omega} |\mathbb A(v)\n|^q\,dS\BBB.
 \end{equation*}
 \MK 
 
 \EM
 The main results about convergence of minimizer of functionals $\mathcal G_\eps$ \AAA with Dirichlet boundary conditions,  that we shall state in the next section, carry over for functionals $\mathcal F_\eps$ and $\mathcal I_\eps$ (we refer to \AAA Corollary \ref{fandi} below\BBB). \AAA On the other hand, when considering the Neumann problem, an interesting difference will arise in the treatment of functional $\mathcal I_\eps$.
 \EOR
 


 \section{Main results}\label{Sec:Results}


Let $d\ge 2$.
The strain energy density $W:\mathbb R^{d\times d}\to [0,+\infty]$ appearing in \eqref{Geps}  is a frame indifferent function that is assumed to be minimized at rotations, and sufficiently smooth around rotations.   Summarizing, the basic assumptions satisfied by $W$ are the following
\begin{equation}\label{frame}\begin{aligned}
&\mbox{$W:\mathbb R^{d\times d}\to [0,+\infty]$ \MK  is  continuous \EOR}, \\
&\mbox{$W(\mathbf R\mathbf F)=W(\mathbf F)$ for every $\mathbf R\in SO(d)$ and every $\mathbf F\in\R^{d\times d}$},\\
& \mbox{$W(\mathbf F)\ge W(\mathbf I)=0$ for every $\mathbf F\in\R^{d\times d}$},\\
& \mbox{$W\in C^2(\mathcal U)$ for some suitable open neighborhood $\mathcal U$ of $SO(d)$ in $\R^{d\times d}$},
\end{aligned}
\end{equation}
where $SO(d)$ denotes the special orthogonal group. A consequence of the smoothness of $W$ around rotations is therefore $DW(\mathbf R)=0$ for every $\mathbf R\in SO(d)$. \MK Continuity of $W$ means that whenever $F_k\to F$ in $\R^{d\times d}$ then $W(F_k)\to W(F)$ as $k\to\infty$ . \EOR
Moreover, $W$ is assumed to satisfy the following \MK convexity   property at the identity: \EOR
\begin{equation}\label{coerc}
\mbox{there exists $C>0$ such that }\;\;\mathbf F^T\,D^2W(\mathbf I)\,\mathbf F\ge C|\mathrm{sym}\,\mathbf F|^2\;\;\mbox{ for every $\mathbf F\in\mathbb R^{d\times d}$},
\end{equation}
where $\mathrm{sym}\,\mathbf F$ denotes the symmetric part of $\mathbf F$ and $|\cdot|$ denotes the Euclidean norm in $\R^{d\times d}$, i.e., $|\mathbf G|^2=\tr(\mathbf G^T\mathbf G)$. Here, $D^2W$ denotes the Hessian of $W$, and  $D^2W(\mathbf I)$ is the fourth order elasticity tensor, appearing in the quadratic potential acting on the infinitesimal strain tensor in the linearized energy $\mathcal G_*$ from \eqref{G*}.
Another standard coercivity condition that we shall use is the following: 
\begin{equation}\label{dist}
\mbox{there exists $\bar C>0$ such that }\;\;W(\mathbf F)\ge \bar C\,\mathrm{dist}^2(\mathbf F,SO(d))\;\;\mbox{ for every $\mathbf F\in\mathbb R^{d\times d}$},
\end{equation}
where $\mathrm{dist}(\mathbf F,SO(d)):=\inf \{|\mathbf F-\mathbf R|: \mathbf R\in SO(d)\}$. We remark that if $W$ satisfies \eqref{frame}, then \eqref{dist} is stronger than \eqref{coerc}. Indeed, since for every $\mathbf G\in\mathbb R^{d\times d}$ with positive determinant there holds $\mathrm{dist}(\mathbf G,SO(d))=|\sqrt{\mathbf G^T\mathbf G}-\mathbf I|$, we see that if \eqref{dist} holds then
\[\begin{aligned}
\frac12\mathbf F^T\,D^2W(\mathbf I)\,\mathbf F&=\lim_{\eps\to 0}\frac1{\eps^2}W(\mathbf I+\eps\mathbf F)\ge\bar C\limsup_{\eps \to0}\frac1{\eps^2}\,\mathrm{dist}^2(\mathbf I+\eps\mathbf F,SO(d))\\&= {\bar C}\limsup_{\eps\to0}\frac1{\eps^2}\,|\sqrt{(\mathbf I+\eps\mathbf F^T)(\mathbf I+\eps\mathbf F)}-\mathbf I|^2=\bar C|\mathrm{sym}\,\mathbf F|^2
\end{aligned}\]
for every $\mathbf F\in\mathbb R^{d\times d}$.
We also remark that the natural conditions  $\det\mathbf F\le 0\Rightarrow W(\mathbf F)=+\infty$ and $W(\mathbf F)\to+\infty$ if $ \det\mathbf F\to 0^+$ are compatible with all the above assumptions, but not necessary for the theory.
\MK

We will further assume that, \MMM for suitable $p>1$, \MK the function $H$ appearing in the second gradient term of \eqref{Geps} satisfies the following: 
\begin{equation}\label{H}\begin{aligned}
& H:\R^{d\times d\times d}\to\R {\textrm{ is a convex  positively $p$-homogeneous function}}, \\
&{\textrm{there exist  $C_0>0$, $C_1>0\;\;$ s.t.  $\;\forall\,\mathbf{B}\in\R^{d\times d\times d}$ }}\;\; C_0|\mathbf{B}|^p\le H(\mathbf{B})\le C_1(1+|\mathbf{B}|^p),\\
& H(\mathbf{RB})=H(\mathbf{B})  \textrm { for every }\mathbf{B} \in\R^{d\times d\times d} \textrm{ and every } \mathbf{R}\in SO(d).
\end{aligned}
\end{equation}
Here the product between $\mathbf{R}$ and $\mathbf{B}$ is defined as $(\mathbf{RB})_{imn}=\mathbf{R}_{ik}\mathbf{B}_{kmn}$ for all $i,m,n\in\{1,\ldots, d\}$.
\EOR

The load functional $\mathcal L$ appearing in $\mathcal G_\eps$ and in $\mathcal G_*$ is assumed to be a linear continuous functional on $W^{2,p}(\Omega;\R^d)$, where $p>1$ is the exponent that appears \MMM in \eqref{H}, \KKK 
 so that there exists a constant $C_{\mathcal L}>0$ such that
\begin{equation}\label{linearload}
|\mathcal L(v)|\le C_{\mathcal L}\,\|v\|_{W^{2,p}(\Omega;\R^d)}\quad\mbox{$\forall\;v\in W^{2,p}(\Omega;\R^d)$}.
\end{equation}

 \EM Let us start by considering the Dirichlet problem. \EOR
Let $\Gamma$ denote a closed subset of $\partial\Omega$ such that $S(\Gamma)=\mathcal H^{d-1}(\Gamma)>0$. For $k\in\mathbb N$ and $p>1$, let  $W^{k,p}_\Gamma(\Omega;\mathbb R^d)$ denote the \EM $W^{k,p}(\Omega;\R^d)$ completion of the space of restrictions to $\Omega$ of $C^{\infty}_c(\R^d\setminus\Gamma;\R^d)$ functions.  \EOR
Taking advantage of the homogeneous Dirichlet boundary condition on $\Gamma$, for $p>1$ we have the following \EM Poincar\'e \EOR inequality: there is $K>0$ such that
\begin{equation}\label{Korn}
\|v\|_{W^{2,p}(\Omega;\R^d)} 
\le 
K\|\nabla^2 v\|_{L^p(\Omega;\R^{d\times d\times d})}\quad\mbox{$\forall\;v\in W^{2,p}_{\Gamma}(\Omega;\R^d)$}.
\end{equation}

The following is the main result  about convergence of minimizers of the Dirichlet problem for functionals $\mathcal G_\eps$. For $p>d/2$ we let $\mathcal G_\eps:W^{2,p}_\Gamma(\Omega;\R^d)\to \R\cup\{+\infty\}$ be defined by \eqref{Geps} and  $\mathcal G_*: W^{2,p}_\Gamma(\Omega;\R^d)\to\R$ be defined by \eqref{G*}. \MMM Existence of minimizers over $W^{2,p}_\Gamma(\Omega;\R^d)$ for functional $\mathcal G_\eps$ (for fixed $\eps$) and functional $\mathcal G_*$ will be preliminarily proved in Section \ref{Sec:Proof1}.

\MK 
\begin{theorem}\label{mainth1} Let \AAA $q\ge1$ and $p\ge dq/(q+1)$. If $d=2$ and $q=1$, assume in addition that $p>1$. \BBB Suppose that $\mathcal L$ is  a bounded linear  functional on $W^{2,p}(\Omega;\R^d)$,  that $W$ satisfies \eqref{frame} and \eqref{coerc}, and that \eqref{H} holds.  Let $(\eps_j)_{j\in\mathbb N}\subset(0,1)$ be a vanishing sequence and let $(v_j)_{j\in\mathbb N}\subset W^{2,p}_\Gamma(\Omega;\mathbb R^d)$ be a sequence of  minimizers for functionals $\mathcal G_{\eps_j}$ over $W^{2,p}_\Gamma(\Omega;\mathbb R^d)$. 
Then the sequence $(v_j)_{j\in\mathbb N}$ is  weakly converging in $W^{2,p}(\Omega;\mathbb R^d)$ to the unique solution to the problem
$$
\min\left\{\mathcal G_*(v): v\in W^{2,p}_\Gamma(\Omega;\R^d)\right\}.$$
\end{theorem}
\EOR

 When considering the pure traction problem, it is natural to assume that external loads are equilibrated, i.e., with null resultant and momentum:  
 \begin{equation}\label{equi}
 \mathcal L(a+\mathbf Wx)=0 \quad\mbox{for every $a\in\mathbb R^d$ and every skew-symmetric matrix $\mathbf W\in\mathbb R^{d\times d}$}.
 \end{equation}
  A consequence of \eqref{equi} is that, by invoking Korn and Poincar\'e inequalities, if $\mathcal L$ is a bounded linear functional over $W^{2,r}(\Omega;\mathbb R^d)$ for some $r>1$, thus satisfying \eqref{linearload} with $p=r$, and if $u\in W^{2,r}(\Omega;\R^d)$, we have for some suitable $a\in\mathbb R^d$ and some suitable skew-symmetric $\mathbf W\in \R^{d\times d}$
\begin{equation}\label{kornpoincare}\begin{aligned}
|\mathcal L(u)|&=|\mathcal L(u-\mathbf Wx-a)|\\&\le C_{\mathcal L}\left(\|u-\mathbf Wx-a\|_{L^r(\Omega;\R^d)}+\|\nabla u-\mathbf W\|_{L^r(\Omega;\mathbb R^{d\times d})}+\|\nabla^2 u\|_{L^r(\Omega;\R^{d\times d\times d})}\right)\\
&\le  C_{\mathcal L}\left((1+\mathfrak c)\mathfrak K \, \|\mathbb E (u)\|_{L^r(\Omega;\mathbb R^{d\times d})}+\|\nabla^2 u\|_{L^r(\Omega;\R^{d\times d\times d})}\right),
\end{aligned}\end{equation}
 where $\mathfrak c$ is the constant in Poincar\'e inequality and $\mathfrak K$ is the constant in second Korn inequality \cite{N}, both depending on $r$ and $\Omega$ only.
 
A further condition that proves to be \AAA crucial \BBB for the theory is
\begin{equation}\label{compatibility}
\mathcal L(\mathbf Rx-x)\le 0\quad\mbox{ for every $\mathbf R\in SO(d)$}
\end{equation}
expressing the fact that external loads have an overall effect of traction (and not of compression) on $\Omega$. This condition appears in \cite{MPTJOTA,MPTARMA,MP2,MPsharp,MOP}.
Following the same references, for load functionals that satisfy \eqref{compatibility} we introduce the \AAA set \BBB
\[
\mathcal S_{\mathcal L}^0:=\{\mathbf R\in SO(d):\mathcal L(\mathbf Rx-x)=0\},
\]
which plays a crucial role in the linearization process. Indeed, if $\mathbf R_0\in SO(d)$ exists such that $\mathcal L(\mathbf R_0x-x)>0$, if $(\eps_j)_{j\in\mathbb N}\subset(0,1)$ is any vanishing sequence, letting $\tilde v_j:=\frac1{\eps_j}(\mathbf R_0x-x)$, since $\cof \mathbf R_0=\mathbf R_0$, $W(\mathbf R_0)=0$ and $\nabla^2(\mathbf R_0x-x)=0$, we see that
\[
\mathcal G_{\eps_j}(\tilde v_j)=-\frac1{\eps_j}\mathcal L(\mathbf R_0x-x),
\]
 thus
 \[
 \lim_{j\to+\infty}\inf_{C^\infty(\overline\Omega;\R^d)}\mathcal G_{\eps_j}=-\infty.
 \]
 Therefore \eqref{compatibility} is a necessary condition for avoiding the nonlinear energies being unbounded from below as the parameter $\eps$ goes to zero. Moreover, under the conditions \eqref{equi} and \eqref{compatibility}, it turns out (see Theorem \ref{main2} below) that the actual limit functional is
 \[
 \overline{\mathcal G}(u):= \frac12\int_\Omega \mathbb E(u)\,\nabla^2 W(\mathbf I)\,\mathbb E(u)\,dx+\int_\Omega H(\nabla^2 v)\,dx -\max_{\mathbf R\in\mathcal S^0_\mathcal L}\mathcal L(\mathbf R v)+\gamma\AAA \int_{\partial \Omega}|\mathbb A(v)\n\cdot\n|^q\,dS\BBB.
 \]
 We may notice that under a stronger condition on external loads, i.e. under the additional assumption $\mathcal S^0_\mathcal L\equiv\{\mathbf I\}$ (meaning that external loads do not have any axis of equilibrium), then $\overline{\mathcal{G}}\equiv\mathcal G_*$. See \cite{MPsharp,MOP} for a thorough discussion on this topic in linearized models with no surface tension effect.

 We have the following second main result.
 It will be proven in Section \ref{Sec:Proof2}, after having proven the existence of minimizers over $W^{2,p}(\Omega;\R^d)$, in the pure traction problem, for functional $\mathcal G_\eps$ (for fixed $\eps$) and for functional $\overline{\mathcal G}$.
 \MK 
 \begin{theorem}\label{main2}
Let \AAA  $q\ge1$, $p\ge dq/(q+1)$. If $d=2$ and $q=1$, assume in addition that $p>1$. \BBB Let $W$ satisfy \eqref{frame} and \eqref{dist} and let $H$ satisfy \eqref{H}. Let $\mathcal L$ be a bounded linear functional over $W^{2,2\wedge p}(\Omega;\R^d)$ that satisfies \eqref{equi} and \eqref{compatibility}.
Let $(\eps_j)_{j\in\mathbb N}\subset(0,1)$  be a vanishing sequence.
 If $(v_j)_{j\in\mathbb N}\subset W^{2,p}(\om,\mathbb R^d)$ is a sequence of minimizers of $\mathcal G_{\eps_j}$ over $W^{2,p}(\Omega;\R^d)$, 
 then there exists a sequence
  $(\mathbf R_j)_{j\in\mathbb N}\subset SO(d)$  such that, by defining 
 \begin{equation*}
u_j(x):= \mathbf R_j^Tv_j(x)+\frac1{\eps_j}(\mathbf R^T_jx-x),
\end{equation*}\KKK
in the limit as $j\to+\infty$, along a  suitable (not relabeled) subsequence, there holds
 \begin{equation*}
 \nabla u_j\to\nabla  u_* \; \mbox{ weakly in $W^{1, p}(\Omega;\R^{d\times d})$},\quad
 \end{equation*}
where $u_*\in W^{2,p}(\Omega,\R^d)$ is a minimizer of $\overline{\mathcal G}$ over $W^{2,p}(\om,\mathbb R^d)$,
and
\begin{equation}\label{last}   \mathcal G_{\eps_j}(v_j)\to \overline{\mathcal G}(u_*),\quad\qquad
\EM \min_{W^{2,p}(\om,\R^d)}\EOR\mathcal G_{\eps_{j}}\to \min_{W^{2,p}(\om,\R^d)}\overline{\mathcal G}. 
\end{equation}
 \end{theorem}
 \EOR
 
 In order to obtain compactness in Theorem \ref{main2}, we shall make use of  the Firesecke-James-M\"uller rigidity inequality \cite{FJM0}, stating that there exists a constant $C_\Omega>0$ (only depending on $\Omega$) such that for every $\varphi\in W^{1,2}(\Omega;\R^d)$ there is $\mathbf R\in SO(d)$ such that
 \begin{equation}\label{rigidity}
 \int_\Omega |\nabla \varphi-\mathbf R|^2\le C_\Omega\int_\Omega\mathrm{dist}^2(\nabla\varphi,SO(d)).
 \end{equation}
In fact, as we shall see from the proof the sequence $(\mathbf R_j)_{j\in\mathbb N}$ in Theorem \ref{main2} can be chosen as a sequence of rotations for which \eqref{rigidity} holds with $\varphi=v_j$ for every $j\in\mathbb N$.

\MMM
We remark that
our main converge results hold true for  other model energies that we have introduced in Section \ref{intro}. Indeed, defining $\overline{\mathcal F}$  by replacing the term $\mathcal L(v)$ with the term $\max_{\mathbf R\in\mathcal S^0_\mathcal L}\mathcal L(\mathbf R v)$ in functional $\mathcal F_*$  we have
\begin{corollary}\label{fandi} The same statement of {\rm Theorem \ref{mainth1}}    holds true  if we replace functionals $\mathcal G_\eps$ and functional $\mathcal G_*$ therein with functionals $\mathcal F_\eps$ (resp. $\mathcal I_\eps$) and functional $\mathcal F_*$ (resp. $\mathcal I_*$). 
\AAA Moreover, the same statement of {\rm Theorem \ref{main2}}    holds true  if we replace functionals $\mathcal G_\eps$ and functional $\overline{\mathcal G}$ therein with functionals $\mathcal F_\eps$  and functional $\overline{\mathcal F}$.\BBB 
\end{corollary}  \KKK
 
 \AAA	
 We conclude with the following result that shows a remarkable difference in the limiting behavior of functionals $\mathcal I_\eps$. The surface live load term prevents indeed rigid rotations. 
  \begin{theorem}\label{main3}
Let \AAA   $p\ge dq/(q+1)$. If $d=2$ and $q=1$, assume in addition that $p>1$.  Let $W$ satisfy \eqref{frame} and \eqref{dist} and let $H$ satisfy \eqref{H}. Let $\mathcal L$ be a bounded linear functional over $W^{2,2\wedge p}(\Omega;\R^d)$ that satisfies \eqref{equi} and \eqref{compatibility}.
Let $(\eps_j)_{j\in\mathbb N}\subset(0,1)$  be a vanishing sequence.
 If $(v_j)_{j\in\mathbb N}\subset W^{2,p}(\om,\mathbb R^d)$ is a sequence of minimizers of $\mathcal I_{\eps_j}$ over $W^{2,p}(\Omega;\R^d)$, 
then in the limit as $j\to+\infty$, along a  suitable (not relabeled) subsequence, there holds
 \begin{equation*}
 \nabla v_j\to\nabla  v_* \; \mbox{ weakly in $W^{1, p}(\Omega;\R^{d\times d})$},\quad
 \end{equation*}
where $v_*\in W^{2,p}(\Omega,\R^d)$ is a minimizer of ${\mathcal I}_*$ over $W^{2,p}(\om,\mathbb R^d)$,
and
\begin{equation}\label{last1}   \mathcal I_{\eps_j}(v_j)\to {\mathcal I}_*(v_*),\quad\qquad
\EM \min_{W^{2,p}(\om,\R^d)}\EOR\mathcal I_{\eps_{j}}\to \min_{W^{2,p}(\om,\R^d)}{\mathcal I}_*. 
\end{equation}
 \end{theorem}
\BBB

\section{The Dirichlet problem: proof of Theorem \ref{mainth1}}\label{Sec:Proof1} 

We first prove existence of minimizers for functional $\mathcal G_*$.  

\begin{lemma}\label{exist}
Let \AAA $p\ge 2d/(d+2)$ and $p\ge dq/(d-1+q)$. If $d=2$ and $q=1$, let $p>1$.  \MMM
Let \eqref{frame}, \eqref{coerc} and \eqref{H} hold. Let $\mathcal L$ be a bounded linear functional over $W^{2,p}(\Omega;\R^d)$. \BBB
There exists a unique solution to the  problem
\[
\min\{\mathcal G_*(v): v\in W^{2,p}_{\Gamma}(\Omega;\R^d)\}.
\]
\end{lemma}
\begin{proof}
The \AAA conditions $d>p\ge\left(2d/(d+2)\right)\vee \left(dq/(d-1+q)\right)$  and $p>1$ imply that ${dp}/{(d-p)}\ge 2$ and $(d-1)p/(d-p)\ge q$, \BBB therefore we have the Sobolev embedding $W^{1,p}(\Omega;\R^d)\hookrightarrow L^2(\Omega;\R^d)$ and the Sobolev trace embedding \AAA $W^{1,p}(\Omega;\R^d)\hookrightarrow L^{q}(\partial\Omega)$. \AAA These embeddings also hold true if $p\ge d$. We deduce that $\mathbb E(v)\in L^2(\Omega;\R^{d\times d})$ and that \AAA $\mathbb A(v)\in L^q(\partial\Omega;\mathbb R^{d\times d})$ \BBB for every  $v\in W^{2,p}(\Omega;\R^d)$.  Thus $\mathcal G_*(v)$ is finite for every $v\in W^{2,p}(\Omega;\R^d)$.

\MK By \eqref{frame}, \eqref{H}, \eqref{linearload}, and \eqref{Korn}, we get
\begin{equation}
\begin{aligned}\label{Lowerbound}
\mathcal G_*(v)&\ge \int_\Omega H(\nabla^2 v)-C_{\mathcal L} \|v\|_{W^{2,p}(\Omega;\R^d)} 
\ge C_0\int_\Omega |\nabla^2 v|^p-C_{\mathcal L} \|v\|_{W^{2,p}(\Omega;\R^d)}\\
&\ge \frac{C_0}{K^p} \|v\|^p_{W^{2,p}(\Omega;\R^d)}-C_{\mathcal L} \|v\|_{W^{2,p}(\Omega;\R^d)}
\end{aligned}
\end{equation}
and the \MMM right hand side is uniformly bounded from below on $ W^{2,p}_\Gamma(\Omega;\R^d)$ because $p>1$.
\EOR

Let $(v_n)_{n\in\mathbb N}\subset W^{2,p}_\Gamma(\Omega;\R^d)$ be a minimizing sequence. The above estimate, along with \eqref{Korn} and $\mathcal G_*(0)=0$, shows that such a sequence is bounded in $W^{2,p}(\Omega;\R^d)$, hence admitting a (not relabeled) weakly converging subsequence, the limit point being denoted by $v$. However, \MMM the term involving $H$ in functional $\mathcal G_*$ is weakly lower semicontinuous  over $W^{2,p}(\Omega;\R^d)$ thanks to \eqref{H}\BBB, as well as the interfacial term, \AAA thanks to the Sobolev trace embedding $W^{1,p}(\Omega;\R^d)\hookrightarrow L^{q}(\partial\Omega;\R^d)$. Since $p>1$ we also have the compactness of the Sobolev trace embedding $W^{1,p}(\Omega;\R^d)\hookrightarrow L^{1}(\partial\Omega;\R^d)$, which implies that  $v\in W^{2,p}_\Gamma(\Omega;\R^d)$. \BBB  On the other hand,   $\mathcal L$ is weakly continuous over $W^{2,p}(\Omega;\R^d)$, and we have $$\liminf_{n\to+\infty}\int_{\R^d}\mathbb E(v_n)\,D^2W(\mathbf I)\,\mathbb E(v_n)\ge\int_{\Omega}\mathbb E(v)\,D^2W(\mathbf I)\,\mathbb E(v)$$
by the embedding $W^{1,p}(\Omega;\R^d)\hookrightarrow L^2(\Omega;\R^d)$ and by the weak $L^2(\Omega;\mathbb R^{d\times d})$ lower semicontinuity of the map $\mathbf G\mapsto\int_\Omega\mathbf G^T\,D^2W(\mathbf I)\,\mathbf G$. Therefore, $\mathcal G_*$ is lower semicontinuous with respect to the weak $W^{2,p}(\Omega;\R^d)$ convergence. The result follows by the direct method of the calculus of variations.
Uniqueness of the minimizer follows from strict convexity of  $\mathcal G_*$.
\end{proof}

The following is a key lemma, providing the rigorous linearization of the interfacial term. 
\begin{lemma}\label{keylemma}
\AAA Let $p\ge dq/(q+1)$. \BBB Let $v\in W^{2,p}(\Omega;\R^d)$.
Let $(\eps_j)_{j\in \mathbb N}\subset(0,1)$ be a vanishing sequence and let $(v_j)_{j\in\mathbb N}\subset W^{2,p}(\Omega;\R^d)$ be a sequence  such that $\nabla v_j$ weakly converge to $\nabla v$ in $W^{1,p}(\Omega;\R^{d\times d})$ as $j\to+\infty$. Then
\begin{equation}\label{mainlimit}
\liminf_{j\to+\infty}\AAA \frac1{\eps_j^q}\int_{\partial \Omega}\left||\cof(\mathbf I+\eps_j\nabla v_j)\,\mathbf n|-1\right|^q\,dS\ge\int_{\partial\Omega}|\mathbb A(v)\mathbf n \cdot\mathbf n|^q\,dS.\BBB
\end{equation}
\AAA If $\nabla v_j$ strongly converge to $\nabla v$ in $W^{1,p}(\Omega;\R^{d\times d})$ as $j\to+\infty$
we also have
\begin{equation}\label{fully}\lim_{j\to+\infty}\AAA \frac1{\eps_j^q}\int_{\partial \Omega}\left||\cof(\mathbf I+\eps_j\nabla v_j)\,\mathbf n|-1\right|^q\,dS=\int_{\partial\Omega}|\mathbb A(v)\mathbf n \cdot\mathbf n|^q\,dS.\BBB
\BBB\end{equation}
\end{lemma}
\begin{proof}
We have the Sobolev trace embedding $W^{1,p}(\Omega;\R^d)\hookrightarrow L^{(d-1)p/(d-p)}(\partial\Omega;\R^d)$ if $d>p$.
Else if $p\ge d$, we have $W^{1,p}(\Omega;\R^d)\hookrightarrow L^{r}(\partial\Omega)$ for every $r\in[1,+\infty)$.
 Since \AAA $p\ge dq/(q+1)$, we have $(d-1)p/(d-p)\ge (d-1)q$. \BBB Since the entries of $\cof\mathbf F$ are polynomials of degree $d-1$ in the entries of $\mathbf F$,  we deduce that $\cof\nabla \psi\in \AAA L^q(\partial\Omega;\R^{d\times d})\BBB $  for every $\psi\in W^{2,p}(\Omega;\R^d)$. This shows that all the integrands appearing in \eqref{mainlimit} are in $L^1(\partial\Omega)$.

By the Cayley-Hamilton formula, 
we have
\begin{align*}
\cof (\mathbf I+\eps_j\nabla v_j)=\mathbf I+\eps_j\mathbb A(v_j)+\sum_{k=2}^{d-1}\eps_j^k\,\mathbb B_k(v_j)
\end{align*}
where $\mathbb B_k(v_j)$ is a matrix whose entries are polynomials of degree $k$ in the entries of $\nabla v_j$, and the sum is understood to be zero if $d=2$. Letting $$\mathbb B(v_j):=\sum_{k=2}^{d-1}\eps_j^{k-2}\mathbb B_k(v_j),$$ we get therefore
\begin{equation}\label{coro2}
\cof (\mathbf I+\eps_j\nabla v_j)=\mathbf I+\eps_j\mathbb A(v_j)+\eps_j^2\,\mathbb B(v_j).
\end{equation}
We notice that
\begin{equation}\label{cofn}\begin{aligned}
|\cof(\mathbf I+\eps_j\nabla v_j)\,\n|^2&=|\mathbf n+\eps_j\mathbb A(v_j)\,\mathbf n+\eps_j^2\,\mathbb B(v_j)\,\mathbf n|^2\\&=1+2\eps_j\mathbb A(v_j)\,\mathbf n\cdot\n+\eps_j^2\,\mathbb D(v_j)\,\n\cdot\n,
\end{aligned}\end{equation}
where
\begin{equation}\label{Dj}
\mathbb D(v_j):=\mathbb A(v_j)^T\mathbb A(v_j)+2\mathbb B(v_j)+2\eps_j\,\mathbb B(v_j)^T\mathbb A(v_j)+\eps_j^2\,\mathbb B(v_j)^T\,\mathbb B(v_j).
\end{equation}

We observe that the following properties hold:
\begin{equation}\label{sobolev}
\mbox{$\nabla v_j\rightharpoonup \nabla v$   \AAA weakly in $L^q(\partial\Omega;\R^{d\times d})$ as $j\to+\infty$\BBB}, 
\end{equation}
\begin{equation}\label{Asobolev}
\mbox{$\mathbb A(v_j)\rightharpoonup \mathbb A(v)$  \AAA weakly in $L^q(\partial\Omega;\R^{d\times d})$ as $j\to+\infty$\BBB},
\end{equation}
\begin{equation}\label{trace}
\mbox{the sequence $(\mathbb B(v_j))_{j\in\mathbb N}$ is bounded in \AAA $L^q(\partial\Omega;\R^{d\times d})$\BBB}.
\end{equation}
Indeed, \eqref{trace} follows from the Sobolev trace embedding $W^{1,p}(\Omega)\hookrightarrow L^{(d-1)p/(d-p)}(\partial\Omega)$, since \AAA $d>p\ge dq/(q+1)$ \BBB implies \AAA $(d-1)p/(d-p)\ge  (d-1)q$, \BBB and since $\mathbb B(v_j)$ is polynomial of degree $d-1$ in the entries of $\nabla v_j$ (and $\eps_j<1$). Else if $p\ge d$ we have the embedding $W^{1,p}(\Omega)\hookrightarrow L^r(\partial\Omega)$ for every $r\in[1,+\infty)$. \AAA For the same reason, \eqref{sobolev} and its direct consequence \eqref{Asobolev} hold true, since $(d-1)q\ge q$. 	\BBB
 In particular, $\nabla v$ has a \AAA $L^q(\partial\Omega)$ \BBB trace on $\partial\Omega$.

We define
\begin{equation}\label{Qj}
Q_j:=\{x\in\partial\om: |\mathbb A(v_j(x))|+\eps_j|\mathbb B(v_j(x))|<2^{-4}\eps_j^{-1/4}\}
\end{equation}
and we notice that 
\[
S(\partial\Omega\setminus Q_j)\le \int_{\partial\Omega\setminus Q_j}16\eps_j^{1/4}(|\mathbb A(v_j)|+\eps_j|\mathbb B(v_j)|)\,dS\le 16\eps_j^{1/4}\int_{\partial\Omega}(|\mathbb A(v_j)|+\eps_j|\mathbb B(v_j)|)\,dS
\]
so that \eqref{Asobolev} and \eqref{trace} imply that $S(\partial\Omega\setminus Q_j)\to 0$ as $j\to+\infty$.
By using \eqref{Qj}, \eqref{Dj} and $\eps_j<1$ it is not difficult to see that on $Q_j$ there hold
\begin{equation}\label{Q1}
|2\eps_j\,\mathbb A(v_j)\,\n\cdot\n+\eps_j^2\,\mathbb D(v_j)\,\n\cdot\n|\le 2\eps_j\,|\mathbb A(v_j)|+\eps_j^2\,|\mathbb D(v_j)|<\frac12 \eps_j^{3/4}<\frac12,
\end{equation}
\begin{equation}\label{Q2}
\eps_j|\mathbb D(v_j)|\le \sqrt{\eps_j}+2\eps_j|\mathbb B(v_j)|.
\end{equation}
Thanks to \eqref{Q1}, starting from \eqref{cofn} and using
  Taylor series, on $Q_j$ we find that
\begin{equation}\label{long}\begin{aligned}
&|\cof(\mathbf I+\eps_j\nabla v_j)\,\n|-1
=\sqrt{1+2\eps_j\mathbb A(v_j)\,\mathbf n\cdot\n+\eps_j^2\,\mathbb D(v_j)\,\n\cdot\n}-1\\
&\qquad=\eps_j\mathbb A(v_j)\,\n\cdot\n+\frac{\eps_j^2}2\mathbb D(v_j)\,\n\cdot\n+\sum_{k=2}^{+\infty}\alpha_k(2\eps_j\mathbb A(v_j)\,\n\cdot\n+\eps_j^2\,\mathbb D(v_j)\,\n\cdot\n)^k,
\end{aligned}\end{equation}
where
$
\alpha_k:=\dfrac{(-1)^{k-1}(2k)!}{4^k(k!)^2(2k-1)} 
$ (in particular we have $\sum_{k=0}^{+\infty}\alpha_{k+2}\,2^{-k}<+\infty$)
and 
\begin{equation}\label{threelines}\begin{aligned}
&\sum_{k=2}^{+\infty}\alpha_k\,(2\eps_j\mathbb A(v_j)\,\n\cdot\n+\eps_j^2\,\mathbb D(v_j)\,\n\cdot\n)^k\\&\quad=
\eps_j^2\,(2\mathbb A(v_j)\,\n\cdot\n+\eps_j\,\mathbb D(v_j)\,\n\cdot\n)^2\sum_{k=0}^{+\infty}\alpha_{k+2}(2\eps_j\,\mathbb A(v_j)\,\n\cdot\n+\eps_j^2\,\mathbb D(v_j)\,\n\cdot\n)^k\\
&\quad\le \eps_j^{2}\,(2\mathbb A(v_j)\,\n\cdot\n+\eps_j\,\mathbb D(v_j)\,\n\cdot\n)^2\,\sum_{k=0}^{+\infty}\alpha_{k+2}\,2^{-k}\le \eps_j^{3/2}\,\sum_{k=0}^{+\infty}\alpha_{k+2}\,2^{-k}.
\end{aligned}\end{equation}
From \eqref{long} we have
\begin{equation*}\begin{aligned}
\bb_{Q_j}\,\frac{|\cof(\mathbf I+\eps_j\nabla v_j)\,\n|-1}{\eps_j}&=\bb_{Q_j}\, \mathbb A(v_j)\,\n\cdot\n+\bb_{Q_j}\,\frac{\eps_j}2\mathbb D(v_j)\,\n\cdot\n\\&\qquad+\bb_{Q_j}\,\frac1{\eps_j}\sum_{k=2}^{+\infty}\alpha_k(2\eps_j\mathbb A(v_j)\,\n\cdot\n+\eps_j^2\,\mathbb D(v_j)\,\n\cdot\n)^k,
\end{aligned}\end{equation*}
where we may notice that the  two terms on the right hand side converge strongly to $0$ in \AAA $L^q(\partial\Omega)$, \BBB thanks to \eqref{trace}, \eqref{Q2} and \eqref{threelines}.
\MMM
Since $S(\partial\Omega\setminus Q_j)\to 0$ and since \eqref{Asobolev} holds, we deduce \AAA by the equiintegrability of $(\mathbb A(v_j))_{j\in\mathbb N}$ that $\bb_{Q_j}\mathbb A(v_j)\,\n\cdot\n\rightharpoonup \mathbb A(v)\,\n\cdot\n$ weakly in $L^q(\partial\Omega)$ \BBB and thus 
\begin{equation}\label{coro}
\bb_{Q_j}\,\frac{|\cof(\mathbf I+\eps_j\nabla v_j)\,\n|-1}{\eps_j}\rightharpoonup \mathbb A(v)\,\n\cdot\n \quad\mbox{\AAA weakly in $L^q(\partial\Omega)$ \BBB as $j\to+\infty$}
\end{equation}
so that \KKK
\begin{equation*}
\liminf_{j\to+\infty}\AAA \frac1{\eps_j^q}\int_{Q_j}||\cof(\mathbf I+\eps_j\nabla v_j)\,\n|-1|^q\,dS\ge \int_{\partial\Omega}|\mathbb A(v)\,\n\cdot\n|^q\,dS.\BBB
\end{equation*}
In order to conclude, we are left to prove that 
\begin{equation}\label{outQj}
\lim_{j\to+\infty}\AAA \frac1{\eps_j^q}\int_{\partial\Omega\setminus Q_j}||\cof(\mathbf I+\eps_j\nabla v_j)\,\n|-1|^q\,dS=0.\BBB
\end{equation}
But \eqref{cofn} yields \[\begin{aligned}
\frac1{\eps_j}\,||\cof(\mathbf I+\eps_j\nabla v_j)\,\n|-1|&=\frac{1}{\eps_j}\,||\n+\eps_j\mathbb A(v_j)\n+\eps_j^2\,\mathbb B(v_j)\,\n|-|\n||\\&\le\frac{1}{\eps_j}\,|\eps_j\,\mathbb A(v_j)\,\n+\eps_j^2\,\mathbb B(v_j)\,\n| \le|\mathbb A(v_j)|+\eps_j|\mathbb B(v_j)|,
\end{aligned}\]
and since $S(\partial\Omega\setminus Q_j)\to 0$, \eqref{outQj} follows from \eqref{Asobolev} and \eqref{trace}.

\AAA Eventually, if we also have strong $W^{1,p}(\Omega;\R^{d\times d})$ convergence of $\nabla v_j$ to $\nabla v$, we obtain strong convergence in \eqref{sobolev} and \eqref{Asobolev}, thus in \eqref{coro}, so that by taking \eqref{outQj} into account we deduce \eqref{fully}. \BBB
\end{proof}

\MMM
\begin{corollary}\label{corolast} \AAA Let $p\ge dq/(q+1)$. \BBB. Let $v\in W^{2,p}(\Omega;\R^d)$.
Let $(\eps_j)_{j\in \mathbb N}\subset(0,1)$ be a vanishing sequence and let $(v_j)_{j\in\mathbb N}\subset W^{2,p}(\Omega;\R^d)$ be a sequence  such that $\nabla v_j$ weakly converge to $\nabla v$ in $W^{1,p}(\Omega;\R^{d\times d})$ as $j\to+\infty$. Then
\begin{equation}\label{secondboundarylimit}
\lim_{j\to+\infty}\AAA \frac1{\eps_j^q}\left|\int_{\partial\Omega}|\cof(\mathbf I+\eps_j\nabla v_j)\,\n|dS-|\partial\Omega|\right|^q=\left|\int_{\partial\Omega}\mathbb A(v)\,\n\cdot\n\,dS\right|^q\BBB
\end{equation}
and
\begin{equation}\label{thirdboundarylimit}
\AAA \liminf_{j\to+\infty}\frac1{\eps_j^q}\int_{\partial\Omega}|(\cof(\mathbf I+\eps_j\nabla v_j)-\mathbf I)\,\n|^q\,dS=\int_{\partial\Omega}|\mathbb A(v)\,\n|^q\,dS. \BBB
\end{equation}
\AAA If $\nabla v_j$ strongly converge to $\nabla v$ in $W^{1,p}(\Omega;\R^{d\times d})$ as $j\to+\infty$, 
we also have
\begin{equation}\label{fourthboundarylimit}
\AAA \lim_{j\to+\infty}\frac1{\eps_j^q}\int_{\partial\Omega}|(\cof(\mathbf I+\eps_j\nabla v_j)-\mathbf I)\,\n|^q\,dS=\int_{\partial\Omega}|\mathbb A(v)\,\n|^q\,dS. \BBB
\end{equation}
\end{corollary}
\begin{proof}
By following the proof of Lemma \ref{keylemma}, with $Q_j$ still defined by \eqref{Qj}, we see that
 that \eqref{secondboundarylimit} follows from \eqref{coro} and \eqref{outQj}. On the other hand, \eqref{coro2} implies that
\[
\frac1{\eps_j}(\cof(\mathbf I+\eps_j\nabla v_j)-\mathbf I)\,\n=\mathbb A(v_j)\,\n+\eps_j\mathbb B(v_j)\,\n
\]
so that \eqref{thirdboundarylimit} follows from \eqref{Asobolev} and \eqref{trace}. \AAA If $\nabla v_j$ strongly converge to $\nabla v$ in $W^{1,p}(\Omega;\R^{d\times d})$ as $j\to+\infty$, then there is strong convergence in \eqref{Asobolev} so that \eqref{fourthboundarylimit} follows. \BBB
\end{proof}

\KKK

We next prove existence of minimizers $\mathcal G_\eps$, for fixed $\eps>0$.

\MK
\begin{lemma}\label{newexist}
\AAA Let $p\ge dq/(q+1)$. If $d=2$ and $q=1$, assume in addition that $p>1$\BBB. Suppose that $\mathcal L$ is  a bounded linear  functional on $W^{2,p}(\Omega;\R^d)$,  that $W$ satisfies \eqref{frame} and \eqref{coerc}, and that \eqref{H} hold.
Then the functional $\mathcal{G}_\varepsilon$ attains a minimum on $W^{2,p}_\Gamma(\O;\R^d)$ for every $\varepsilon>0$.
\end{lemma}

\begin{proof} We assume wlog that $\eps=1$.
\MMM For every $v\in W^{2,p}_\Gamma(\Omega;\R^d)$, by \eqref{H},  \eqref{linearload} and \eqref{Korn} there holds
\[
\mathcal G_1(v)\ge C_0\int_\Omega|\nabla^2 v|^p-\mathcal L(v)\ge C_0\int_\Omega|\nabla^2 v|^p-K C_{\mathcal L}\|\nabla^2v\|_{L^p(\Omega;\R^{d\times d\times d})}.
\]
 Similarly to the proof of Lemma \ref{exist}, \AAA since $p>1$ \BBB this estimate implies boundedness from below of functional $\mathcal G_1$ on $W^{2,p}_\Gamma(\Omega;\R^d)$,  
and since $\mathcal G_1(0)=0$, it implies that any minimizing sequence $(v_n)_{n\in\mathbb N}\subset W^{2,p}_\Gamma(\Omega;\R^d)$ is bounded in $W^{2,p}(\Omega;\R^d)$, thus  weakly converging up to subsequences to some $v\in W^{2,p}_\Gamma(\Omega;\R^d)$. This yields up to subsequences pointwise a.e. convergence of $\nabla v_n$ to $\nabla v$, so that since $W$ is continuous and nonnegative, Fatou Lemma implies lower semicontinuity of the term involving $W$ in functional $\mathcal G_1$ along the sequence $(v_n)_{n\in\mathbb N}$.   
On the other hand, $\mathcal L$ is obviously continuous along such sequence, while
\[
\liminf_{n\to+\infty}\int_\Omega H(\nabla^2v_n)\ge\int_\Omega H(\nabla^2v)
\]
is a consequence of the convexity of $H$ from \eqref{H}.
 Finally, 
 \AAA  since $p>1$, the Sobolev trace embedding $W^{1,p}(\Omega;\R^d) \hookrightarrow L^1(\partial\Omega;\R^d)$ \BBB is compact, thus we have pointwise $S-a.e.$ convergence of $\nabla v_n$ to $\nabla v$ on $\partial\Omega$ (up to a subsequences), and  by Fatou Lemma we get lower semicontinuity of the interfacial term of functional $\mathcal G_1$. \BBB
\end{proof}
\EOR

By means of the next two lemmas, we prove the $\Gamma$-convergence

\begin{lemma}[\textbf{$\Gamma$-limsup}] \label{limsup} \MMM \AAA Let $p\ge dq/(q+1)$. \BBB Suppose that $\mathcal L$ is  a bounded linear  functional on $W^{2,p}(\Omega;\R^d)$,  that $W$ satisfies \eqref{frame} and \eqref{coerc}, and that \eqref{H} holds. \KKK 
Let $(\eps_j)_{j\in\mathbb N}\subset(0,1)$ be a vanishing sequence.
Let $v\in W^{2,p}_\Gamma(\Omega;\R^d)$. There exists a sequence $(v_j)_{j\in\mathbb N}\subset C^{\infty}(\overline \Omega;\R^d)\cap W^{2,p}_\Gamma(\Omega;\R^d)$ such that
\[
v_j\to v\quad\mbox{strongly in $W^{2,p}(\Omega;\R^d)$ as $j\to+\infty$}
\]
and
\[
\lim_{j\to+\infty}\mathcal G_{\eps_j}(v_j)=\mathcal G_*(v).
\]
\end{lemma}
\begin{proof}
By the density of $C^{\infty}(\overline \Omega;\R^d)\cap W^{2,p}_\Gamma(\Omega;\R^d)$ in $W^{2,p}_\Gamma(\Omega;\R^d)$, with respect to the $W^{2,p}(\Omega,\R^d)$ convergence, there exists a sequence $(\tilde v_j)_{j\in\mathbb N}\subset C^{\infty}(\overline \Omega;\R^d)\cap W^{2,p}_\Gamma(\Omega;\R^d)$ such that $\tilde v_j\to v$ strongly in $W^{2,p}(\Omega;\R^d)$ as $j\to+\infty$.
We suppose wlog that the sequence $(\|\nabla \tilde v_j\|_{C^0(\overline \Omega;\R^{d\times d})})_{j\in\mathbb N}$ is nondecreasing. We notice that it is converging if $p>d$, thanks to the Sobolev embedding $W^{1,p}(\Omega;\R^d)\hookrightarrow C^0(\overline \Omega;\R^d)$ holding in such a case. In any case, if it is converging we let $v_j=\tilde v_j$ for every $j\in\mathbb N$. Else if it is diverging,  we define a new sequence  sequence $(v_j)_{j\in\mathbb N}\subset C^{\infty}(\overline \Omega;\R^d)\cap W^{2,p}_\Gamma(\Omega;\R^d)$  as follows:
\[
\{\tilde v_1,\tilde v_1,\,\ldots, \tilde v_1,\;\tilde v_2,\tilde v_2,\,\ldots,\tilde v_2,\;\tilde v_3,\tilde v_3,\,\ldots,\tilde v_3,\;\tilde v_4,\tilde v_4\ldots\},
\] 
where each of the $\tilde v_k$'s is repeated $t_k$ times, where $t_1$ is defined as the smallest positive integer such that $$\eps_{i}\|\nabla\tilde v_2\|_{C^0(\overline\Omega;\R^{d\times d})}<\frac12\quad\mbox {for every integer $i\ge1+t_1$}, $$ whose existence is ensured by the fact that the sequence $(\eps_j)_{j\in\mathbb N}$ is vanishing, and where
$t_k$ is then recursively defined as the smallest positive integer such that
\[
\eps_i\|\nabla\tilde v_{k+1}\|_{C^0(\overline\Omega;\R^{d\times d})}<\frac1k\quad\mbox{ for every integer $i\ge 1+t_1+\ldots+t_k$}.
\]
It is clear that we have $v_j\to v$ strongly in $W^{2,p}(\Omega;\R^d)$ as $j\to+\infty$. Moreover, by construction we also have 
\begin{equation}\label{C0}
\lim_{j\to+\infty} \eps_j\|\nabla v_j\|_{C^0(\overline\Omega;\R^{d\times d})}=0.
\end{equation}
 
We notice that the assumptions on the strain energy density $W$ imply that
\begin{equation*}
\left|W(\mathbf I+\mathbf F)-\frac12 \mathbf F^T\,D^2W(\mathbf I)\,\mathbf F\right|\le \omega(|\mathbf F|)\,|\mathbf F|^2
\end{equation*}
for every $\mathbf F\in\mathcal {\tilde U}\subset\subset\mathcal U$, where $\omega:[0,+\infty)\to[0,+\infty)$ denotes the modulus of uniform continuity of $D^2W$ on $\mathcal {\tilde U}$, which is a continuous nondecreasing function that vanishes at $0$. But  \eqref{C0} implies that, for every large enough $j$, we have $\eps_j\nabla v_j(x)\in\mathcal {\tilde U}$ for every $x\in\overline\Omega$. Therefore, if $j$ is large enough, we get
\begin{equation}\label{modulus}
\left|\frac1{\eps_j^2}W(\mathbf I+\eps_j\nabla v_j(x))-\frac12 \mathbb E(v_j(x))\,D^2W(\mathbf I)\,\mathbb E(v_j(x))\right|\le \omega(\eps_j|\nabla v_j(x)|)\,|\nabla v_j(x)|^2
\end{equation}
for every $x\in\overline\Omega$, having used the fact that by frame indifference $D^2 W(\mathbf I)$ acts as a quadratic form only on the symmetric part of $\nabla v_j$.

We notice that, since \AAA $p\ge dq/(q+1)\ge d/2$, \BBB if $d>p$ we have $p_*:=dp/(d-p)\ge2$, therefore we have the  embedding $W^{1,p}(\Omega;\mathbb R^d)\hookrightarrow L^2(\Omega)$, still holding if $p\ge d$. Therefore $\mathbb E (v_j)\to\mathbb E( v)$ strongly in $L^2(\Omega;\R^{d\times d})$. Hence,
\[
\begin{aligned}
&\int_\Omega\left|\frac1{\eps_j^2}W(\mathbf I+\eps_j\nabla v_j)-\frac12\mathbb E(v)\,D^2W(\mathbf I)\,\mathbb E(v)\right|\le \int_\Omega\left|\frac1{\eps_j^2}W(\mathbf I+\eps_j\nabla v_j)-\frac12\mathbb E(v_j)\,D^2W(\mathbf I)\,\mathbb E(v_j)\right|
\\
&\qquad\quad+\int_\Omega\left|\frac12\mathbb E(v_j)\,D^2W(\mathbf I)\,\mathbb E(v_j)-\frac12\mathbb E(v)\,D^2W(\mathbf I)\,\mathbb E(v)\right|
\end{aligned}
\]
and the second term in the right hand side goes to $0$ as $j\to+\infty$. On the other hand, the first term in the right hand side gets estimated by means of \eqref{modulus} as
\[
\int_\Omega\left|\frac1{\eps_j^2}W(\mathbf I+\eps_j\nabla v_j)-\frac12\mathbb E(v_j)\,D^2W(\mathbf I)\,\mathbb E(v_j)\right|\le \omega(\eps_j\|\nabla v_j\|_{C^0(\overline\Omega;\R^{d\times d})})\int_\Omega|\nabla v_j|^2
\]
and so it vanishes as well as $j\to+\infty$, thanks to \eqref{C0} and to the boundedness of $(\nabla v_j)_{j\in\mathbb N}$ in $L^2(\Omega;\R^{d,\times d})$, since $\omega(t)\to 0$ as $t\to 0$.
We conclude that 
\[
\lim_{j\to+\infty}\frac1{\eps_j^2}\int_\Omega W(\mathbf I+\eps_j\nabla v_j)=\frac12\int_\Omega \mathbb E(v)\,D^2W(\mathbf I)\,\mathbb E(v).
\]

With respect to the strong $W^{2,p}$ convergence, the interfacial term is continuous by Lemma \ref{keylemma}, while  the load term  and the second gradient term are obviously continuous. The proof is concluded. 
\end{proof}

\begin{lemma}[\textbf{$\Gamma$-liminf}] \label{liminf} \MMM \AAA Let $p\ge dq/(q+1)$. \BBB Suppose that $\mathcal L$ is  a bounded linear  functional on $W^{2,p}(\Omega;\R^d)$,  that $W$ satisfies \eqref{frame} and \eqref{coerc}, and that \eqref{H} holds. \KKK
Let $(\eps_j)_{j\in\mathbb N}\subset(0,1)$ be a vanishing sequence.
Let $v\in W^{2,p}_\Gamma(\Omega;\R^d)$. Let $(v_j)_{j\in\mathbb N}\subset W^{2,p}_\Gamma(\Omega;\R^d)$ be a sequence that weakly converges to $v$ in $W^{2,p}(\Omega;\R^d)$. Then
\[
\liminf_{j\to+\infty}\mathcal G_{\eps_j}(v_j)\ge\mathcal G_*(v).
\]
\end{lemma}
\begin{proof}
Similarly to the proof of Lemma \ref{limsup}, 
we have the  embedding of $W^{1,p}(\Omega)$ in $L^2(\Omega)$. Therefore $\nabla v_j\to\nabla v $ \AAA weakly \BBB in $L^2(\Omega;\R^{d\times d})$.
Let 
\[
H_j:=\{x\in\Omega: \sqrt{\eps_j}|\nabla v_j(x)|<1\},
\]
so that 
\[
|\Omega\setminus H_j|\le \int_{\Omega\setminus H_j}\sqrt{\eps_j}|\nabla v_j|\le\sqrt{\eps_j}\int_\Omega|\nabla v_j|,
\]
thus $|\Omega\setminus H_j|\to 0$ as $j\to+\infty$. For every $x\in H_j$ we have $\eps_j|\nabla v_j(x)|<\sqrt{\eps_j}$ and thus \eqref{modulus} holds for every large enough $j$, yielding
\[\begin{aligned}
&\liminf_{j\to+\infty}\frac1{\eps_j^2}\int_\Omega W(\mathbf I+\eps_j\nabla v_j)\ge\liminf_{j\to+\infty}\frac1{\eps_j^2}\int_{H_j}W(\mathbf I+\eps_j\nabla v_j)\\&\qquad\ge \liminf_{j\to+\infty}\left(\frac12\int_{H_j}\mathbb E(v_j)\,D^2W(\mathbf I)\,\mathbb E(v_j)-\int_{H_j}\omega(\eps_j|\nabla v_j|)\,|\nabla v_j|^2\right)\\&\qquad \ge \liminf_{j\to+\infty}\left(\frac12\int_{H_j}\mathbb E(v_j)\,D^2W(\mathbf I)\,\mathbb E(v_j)-\int_{H_j}\omega(\sqrt{\eps_j})|\nabla v_j|^2\right).
\end{aligned}\]
But $\omega(\sqrt{\eps_j})\to 0$ as $j\to+\infty$ and $(\nabla v_j)_{j\in\mathbb N}$ is bounded in $L^2(\Omega;\R^{d\times d})$, therefore
\[
\liminf_{j\to+\infty}\frac1{\eps_j^2}\int_\Omega W(\mathbf I+\eps_j\nabla v_j)\ge \liminf_{j\to+\infty}\frac12\int_{H_j}\mathbb E(v_j)\,D^2W(\mathbf I)\,\mathbb E(v_j)\ge \frac12\int_{\Omega}\mathbb E(v)\,D^2W(\mathbf I)\,\mathbb E(v),
\]
\AAA
where the last inequality is due to the \AAA weak convergence of $\nabla v_j$ to $\nabla v$ in $L^2(\Omega;\R^{d\times d})$ and to the fact that $|\Omega\setminus H_j|\to 0,$ yielding the weak $L^2(\Omega;\R^{d\times d})$ convergence of $\bb_{H_j}\mathbb E(v_j)$ to $\mathbb E(v)$, and to  the weak $L^2(\Omega;\mathbb R^{d\times d})$ lower semicontinuity of the map $\mathbf G\mapsto\int_\Omega\mathbf G^T\,D^2W(\mathbf I)\,\mathbf G$. \BBB

With respect to the weak $W^{2,p}$ convergence, the interfacial term is \AAA lower semicontinuous \BBB by Lemma \ref{keylemma}, the load term is continuous, and the second gradient term is lower semicontinuous \MMM thanks to the assumptions \eqref{H}. \BBB The proof is concluded. 
\end{proof}

\begin{proofad}
\MK
Since $W\ge 0$, by Young inequality along with \eqref{linearload}, \eqref{H} and \eqref{Korn}, calculations similar to \eqref{Lowerbound} 
 show that
 $\mathcal G_{\eps_j}$ is bounded from below on $W^{2,p}_\Gamma(\Omega;\R^d)$, uniformly with respect to $j\in\mathbb N.$
Moreover, by applying the same estimate to $v_j$, since $\mathcal G_{\eps_j}(0)=0$, we obtain
\[
C_0\int_\Omega |\nabla^2v_j|^p\le \mathcal G_{\eps_j}(0)+1=1
\EOR
\] 
for every large enough $j$, thus showing that the sequence $(v_j)_{j\in\mathbb N}$ is uniformly bounded in $W^{2,p}(\Omega;\R^d)$.
Having shown the $\Gamma$-convergence by means of Lemma \ref{limsup} and Lemma \ref{liminf}, the proof concludes. Indeed, let $v\in W^{2,p}_\Gamma(\Omega;\R^d)$ be a weak $W^{2,p}(\Omega;\R^d)$ limit point of the sequence $(v_j)_{j\in\mathbb N}$. Let $\hat v\in W^{2,p}_\Gamma(\Omega;\R^d)$ and let $(\hat v_j)_{j\in\mathbb N}\subset W^{2,p}_\Gamma(\Omega)$ be a sequence such that $\mathcal G_{\eps_j}(\hat v_j)\to\mathcal G_*(\hat v)$ as $j\to+\infty$, whose existence is ensured by Lemma \ref{limsup}. By minimality of $v_j$  and Lemma \ref{liminf} we get
\[
\mathcal G_*(v)\le\liminf_{j\to+\infty}\mathcal G_{\eps_j}(v_j)\le \limsup_{j\to+\infty}\mathcal G_{\eps_j}(\hat v_j)=\mathcal G_{*}(\hat v).
\]
The arbitrariness of $\hat v\in W^{2,p}_\Gamma(\Omega;\R^d)$ ends the proof.
\end{proofad}


 \section{The  traction problem: proof of Theorem \ref{main2}}\label{Sec:Proof2}
 
We start by proving existence of minimizers for functional $\overline{\mathcal G}$ over $W^{2,p}(\Omega;\R^d)$. This is done after a preliminary lemma

\begin{lemma}\label{aux} Let $p>1$. Let $(u_j)_{j\in\mathbb N}\subset W^{2,p}(\Omega;\R^d)$ be a sequence such that $\mathbb E(u_j)\in L^2(\Omega;\R^{d\times d})$ for every $j\in\mathbb N$. Suppose that   
 \begin{equation}\label{doublebound}\sup_{j\in\mathbb N}\|\mathbb E(u_j)\|_{L^2(\Omega;\R^{d\times d})}+\sup_{j\in\mathbb N}\|\nabla^2u_j\|_{L^p(\Omega;\R^{d\times d\times d})}<+\infty.\end{equation}  Then there exist $u\in W^{2,p}(\Omega;\R^d)$ \MMM with $\mathbb E(u)\in L^2(\Omega;\R^{d\times d})$\BBB, a sequence $(a_j)_{j\in\mathbb N}\subset\R^d$ and a sequence of skew-symmetric matrices $(\mathbf W_j)_{j\in\mathbb N}\subset\mathbb R^{d\times d}$ such that as $j\to+\infty$, along a suitable not relabeled subsequence,
\begin{equation}\label{nuova}\begin{aligned}
&u_j-\mathbf W_jx-a_j\rightharpoonup u\;\;\mbox{weakly  in $W^{2,p}(\Omega;\R^d)$}\\ &\mathbb E(u_j)\rightharpoonup\mathbb E(u)\;\;\mbox{weakly in $L^2(\Omega;\R^{d\times d})$}.\end{aligned}
\end{equation}
\end{lemma} 
 \begin{proof} Assume first that $1<p\le2$. By Poincar\'e inequality and by second Korn inequality, for every $j\in\mathbb N$ there exist $a_j\in\mathbb R^d$ and a skew-symmetric matrix $\mathbf W_j\in\mathbb R^{d\times d}$ such that
 \begin{equation}\label{eg2}
 \|u_j-\mathbf W_jx-a_j\|_{L^{p}(\Omega;\R^d)}\le \mathfrak q\,\|\nabla u_j-\mathbf W_j \|_{L^p(\Omega;\R^{d\times d})}\le \mathfrak q\mathfrak K \|\mathbb E(u_j)\|_{L^p(\Omega;\R^{d\times d})},
 \end{equation}
 where $\mathfrak q$, $\mathfrak K$ are positive constants, only depending on $\Omega$ and $p$.  We conclude from \eqref{doublebound} that there exists $u\in W^{1,p}(\Omega;\R^d)$ such that, by passing to a not relabeled subsequence,  $u_j-\mathbf W_jx-a_j\rightharpoonup u$ weakly in $W^{1,p}(\Omega;\R^{d})$ as $j\to+\infty$. \MMM But then \eqref{doublebound} implies that $\mathbb E(u)\in L^2(\Omega;\R^{d\times d})$ and that $\mathbb E(u_j)\to\mathbb E(u)$ weakly in $L^2(\Omega;\R^{d\times d})$. The boundedness in $L^p(\Omega;\R^{d\times d\times d})$ of the sequence $(\nabla^2u_j)_{j\in\mathbb N}$, still given by \eqref{doublebound}, allows to conclude.
 
 Else if $p\ge2$,  Poincar\'e inequality yields for every $j\in\mathbb N$ the existence  of $\mathbf{U}_j\in\mathbb R^{d\times d}$  such that
 \begin{equation}\label{henri}\begin{aligned}
 \|\mathbb E(u_j)-\mathbf{U}_j\|_{L^2(\Omega;\R^{d\times d})}&\le |\Omega|^{\tfrac{p-2}{2p}}\|\mathbb E(u_j)-\mathbf{U}_j\|_{L^p(\Omega;\R^{d\times d})}\\&\le \mathfrak q|\Omega|^{\tfrac{p-2}{2p}}\|\nabla\mathbb E(u_j)\|_{L^p(\Omega;\R^{d\times d})}\le \mathfrak q|\Omega|^{\tfrac{p-2}{2p}}\|\nabla^2u_j\|_{L^p(\Omega;\R^{d\times d})}.\end{aligned}
 \end{equation}
 But \eqref{doublebound} and \eqref{henri} imply that the sequence $(\mathbf{U}_j)_{j\in\mathbb N}$ is uniformly bounded in $\mathbb R^{d\times d}$. Therefore still by \eqref{henri} we get that $(\mathbb E(u_j))_{j\in\mathbb N}$ is uniformly bounded also in $L^p(\Omega;\R^{d\times d})$,
 so that we still have, by Korn and Poincar\'e inequalities, the validity of \eqref{eg2}, where the right hand side is uniformly bounded with respect to $j$. We thus conclude as done in the case $1<p\le 2$.
  \end{proof}
 
\begin{lemma}\label{barg}
 Let \AAA $p\ge 2d/(d+2)$ and $p\ge dq/(d-1+q)$. If $d=2$ and $q=1$, let $p>1$.  \MMM
 Suppose that $W$ satisfies \eqref{frame} and \eqref{coerc} and that $H$ satisfies \eqref{H}. Let $\mathcal L$ be a bounded linear functional over $W^{2,p\wedge 2}(\Omega;\R^d)$ that satisfies \eqref{equi}. Then, there exists a solution \AAA to each of the problems $$\min\{\overline{\mathcal G}(u):u\in W^{2,p}(\Omega;\mathbb R^d)\}\qquad\mbox{and}\qquad\min\{{\mathcal G}(u):u\in W^{2,p}(\Omega;\mathbb R^d)\}.$$\end{lemma}
\begin{proof}
\AAA We consider the problem for functional $\overline{\mathcal G}$, the proof for the other one being analogous. \BBB
As seen in the proof of Lemma \ref{exist} we have
 the Sobolev embedding $W^{1,p}(\Omega;\R^d)\hookrightarrow L^2(\Omega;\R^d)$ and the Sobolev trace embedding  \AAA $W^{1,p}(\Omega;\R^d)\hookrightarrow L^q(\partial\Omega;\R^d)$, \BBB so that $\overline{\mathcal G}(v)$ is well-defined and finite for every $v\in W^{2,p}(\Omega;\R^d)$.
For every $v\in W^{2,p}(\Omega;\R^d)$ and every $\mathbf R\in \mathcal S^0_\mathcal L$ we define
\begin{equation*}\mathcal H(v,\mathbf R):= \frac12\int_\Omega\mathbb E(v)\,D^2W(\mathbf I)\,\mathbb E(v)+\int_\Omega H(\nabla^2 v)\AAA+\gamma\int_{\partial\Omega}|\mathbb A(v)\,\n\cdot\n|^q\,dS\BBB -\mathcal L(\mathbf R v).\end{equation*}
Let $(u_j,\mathbf R_j)_{j\in\mathbb N}\subset W^{2,p}(\Omega;\R^d)\times\mathcal S^0_\mathcal L$ be a minimizing sequence for 
the minimization problem $$\min\{\mathcal H(u,\mathbf R): (u,\mathbf R)\in W^{2,p}(\Omega;\R^d)\times\mathcal S^0_\mathcal L\}.$$
Since ${\mathcal H}(0,\mathbf I)=0$ we may assume wlog that ${\mathcal H}(u_j,\mathbf R_j)\le 1$ for every $j\in\mathbb N$.
\MMM By \eqref{coerc} and \eqref{H}, by H\"older inequality and by and \eqref{kornpoincare} with $r=2\wedge p$, we get
\[\begin{aligned}
C_0\int_{\Omega}|\nabla^2 u_j|^p+C\int_\Omega |\mathbb E(u_j)|^2&\le 1+\mathcal L(\mathbf R_j u_j)\\&\le1+ Q \left(\|\mathbb E(u_j)\|_{L^2(\Omega;\R^{d\times d})}+\|\nabla^2u_j\|_{L^p(\Omega;\R^{d\times d\times d})}\right),
\end{aligned}\]\BBB
where $Q$ is a suitable positive constant, only depending on $p$,  $\Omega$ and $\mathcal L$. An application of Young inequality in the right hand side (similarly to the proof of Lemma \ref{exist}), shows that the sequence $(u_j)_{j\in\mathbb N}$ satisfies \eqref{doublebound}. The same computation shows that the sequence $(\overline{\mathcal G}(u_j))_{j\in\mathbb N}$ is bounded from below. By Lemma \ref{aux},  there exist $u\in W^{2,p}(\Omega;\R^d)$, a sequence $(a_j)_{j\in\mathbb N}\subset\R^d$ and a sequence of skew-symmetric matrices $(\mathbf W_j)_{j\in\mathbb N}\subset\mathbb R^{d\times d}$ such that \eqref{nuova} holds along a suitable not relabeled subsequence. Since $\mathcal L$ is a bounded linear functional over $W^{2,p}(\Omega;\R^d)$, and since \eqref{equi} holds, we deduce
\[
\lim_{j\to+\infty}\mathcal L(u_j)=\lim_{j\to+\infty}\mathcal L(u_j-\mathbf W_jx-a_j)=\mathcal L(u).
\]
 The Sobolev trace embedding \AAA $W^{1,p}(\Omega)\hookrightarrow L^q(\partial\Omega)$ shows that $\mathbb A(u_j)\,\n\cdot\n\rightharpoonup\mathbb A(u)\,\n\cdot\n$ weakly in $L^q(\partial\Omega)$, therefore  \[
\liminf_{j\to+\infty}\int_{\partial\Omega}|\mathbb A(u_j)\,\n\cdot\n|^q\,dS=\liminf_{j\to+\infty}\int_{\partial\Omega}|\mathbb A(u_j-\mathbf W_jx-a_j)\,\n\cdot\n|^q\,dS\ge\int_{\partial\Omega}|\mathbb A(u)\,\n\cdot\n|^q\,dS.\] \BBB
Along the sequence $(u_j)_{j\in\mathbb N}$, the first two terms of $\mathcal H$ are lower semicontinuous thanks to \eqref{nuova}.
By possibly extracting a further not relabeled subsequence we have $\mathbf R_j\to\mathbf R\in \mathcal S^0_\mathcal L$ as $j\to+\infty$ and then $\mathcal L(\mathbf R_ju_j)\to\mathcal L(\mathbf R u)$.
 We conclude that
\[
\overline{\mathcal G}(u)\le \mathcal H(u,\mathbf R)\le\liminf_{j\to+\infty}\mathcal H(u_j,\mathbf R_j)=\inf_{W^{2,p}(\Omega;\R^d)\times\mathcal S^0_\mathcal L}\mathcal H=\inf_{W^{2,p}(\Omega;\R^d)}\overline{\mathcal G},
\]
thus showing that $u$ is a minimizer for $\overline{\mathcal G}$ over $W^{2,p}(\Omega;\R^d)$.
\end{proof} 

\MMM
Next we prove existence of minimizers over $W^{2,p}(\Omega;\R^d)$ for functional $\mathcal G_\eps$, for fixed $\eps$.
\begin{lemma}\label{nonlineare}
\AAA Let $p\ge dq/(q+1)$. If $d=2$ and $q=1$, assume in addition that $p>1$\BBB. Let $\mathcal L$ be  a bounded linear  functional on $W^{2,p\wedge 2}(\Omega;\R^d)$ that satisfies \eqref{equi}. Suppose that   $W$ satisfies \eqref{frame} and \eqref{dist}, and that $H$ satisfies \eqref{H}.
Then the functional $\mathcal{G}_\varepsilon$ attains a minimum on $W^{2,p}(\O;\R^d)$, for every $\varepsilon>0$.
\end{lemma}
\begin{proof} We fix wlog $\eps=1$. For every $v\in W^{2,p}(\Omega;\R^d)$, the conditions on $p$ imply that $v\in W^{1,2}(\Omega;\R^d)$ by Sobolev embedding, and thanks to \eqref{rigidity} and \eqref{dist} we get
\[
\mathcal G_1(v)\ge \int_\Omega W(\mathbf I+\nabla v)+\int_\Omega H(\nabla^2 v)-\mathcal L(v)\ge \frac{\bar C}{C_\Omega}\int_\Omega|\mathbf I+\nabla v-\mathbf R|^2 +\int_\Omega H(\nabla^2 v)-\mathcal L(v),
\]
for some suitable $\mathbf R\in SO(d)$, depending on $v$, therefore there exists $c>0$ (only depending on $\bar C$, $C_\Omega$ and $d$), such that, also using \eqref{H},
\[
c+\mathcal G_1(v)\ge \int_\Omega |\nabla v|^2+\int_\Omega H(\nabla^2 v)-\mathcal L(v)\ge \int_\Omega |\nabla v|^2+C_0\int_\Omega |\nabla^2 v|^p-\mathcal L(v).
\]
Similarly to the proof of Lemma \ref{barg},
the latter estimate can be combined with    \eqref{kornpoincare} with $r=2\wedge p$ and with Young inequality to obtain that $\mathcal G_1$ is bounded from below over $W^{2,p}(\Omega;\R^d)$ and that every minimizing sequence $(v_n)_{n\in\mathbb N}\subset W^{2,p}(\Omega;\R^d)$ of functional $\mathcal G_1$  is such that $(\nabla v_n)_{n\in\mathbb N}$ is bounded in $L^2(\Omega;\R^{d\times d})$ and such that $(\nabla^2 v_n)_{n\in\mathbb N}$ is bounded in $L^p(\Omega;\R^{d\times d\times d})$. By arguing as in the proof of Lemma \ref{barg}, we may use Poincar\'e inequality and deduce that there exists $v\in W^{2,p}(\Omega;\R^d)$ with $\nabla v\in L^2(\Omega;\R^{d\times d})$ and  a sequence $(a_n)_{n\in\mathbb N}\subset \R^d$ such that, as $n\to+\infty$ along a suitable subsequence, $v_n-a_n\rightharpoonup v$ weakly in $W^{2, p}(\Omega;\R^d)$. 
In particular, up to extraction of a further subsequence, $\nabla v_n\to\nabla v$ a.e. in $\Omega$, thus by Fatou lemma and the continuity of $W$ we obtain the lower semicontinuity of the integral involving $W$ along the sequence $(v_n)_{n\in\mathbb N}$, while the lower semicontinuity of the term involving $H$ is ensured by \eqref{H}.  We have by \eqref{equi} $\mathcal L(v_n)=\mathcal L(v_n-a_n)\to \mathcal L(v)$ as $n\to+\infty$ since $\mathcal L$ is a bounded linear functional over $W^{2,p}(\Omega;\R^d)$. Finally, the interfacial term of functional $\mathcal G_1$ is \AAA lower semicontinuous \BBB along the sequence $(v_n)_{n\in\mathbb N}$, by the same argument in the proof of Lemma \ref{newexist}.  
We conclude that $v$ is a minimizer of $\mathcal G_1$ over $W^{2,p}(\Omega;\R^d)$.
\end{proof}
\BBB
 
 We need three lemmas in order to prove Theorem \ref{main2}, proving respectively compactness, lower bound and upper bound.

 \begin{lemma}\label{compactness2}  \AAA Let $p\ge dq/(q+1)$. If $d=2$ and $q=1$, assume in addition that $p>1$\BBB. Let $M>0$. Let $W$ satisfy \eqref{frame} and \eqref{dist}. Let $\mathcal L$ be a bounded linear functional over $W^{2,2\wedge p}(\Omega;\R^d)$ that satisfies \eqref{equi} and \eqref{compatibility}. Let $(\eps_j)_{j\in\mathbb N}\subset(0,1)$ be a vanishing sequence. Let $(v_j)_{j\in\mathbb N}\subset W^{2,p}(\Omega;\R^d)$ be a sequence such that $\mathcal G_{\eps_j}(v_j)\le M$ for every $j\in\mathbb N$. Then there exist $\mathbf R\in \mathcal S_{\mathcal L}^0$, a sequence $(\mathbf R_j)_{j\in\mathbb N}\subset SO(d)$,  and $u\in W^{2,p}(\Omega;\R^d)$ such that, letting 
  \begin{equation}\label{UJ}
 u_j:=\mathbf R_j^T v_j+\frac{\mathbf R_j^Tx-x}{\eps_j},
 \end{equation}
 in the limit as $j\to+\infty$ (possibly along a not relabeled subsequence) there hold 
 \[
\mathbf R_j\to\mathbf R\qquad\mbox{and}\qquad \nabla u_j\rightharpoonup\nabla  u \;\; \mbox{ weakly in $W^{1, p}(\Omega;\R^{d\times d})$}.
 \]
  \end{lemma}
 \begin{proof}
 A consequence of \eqref{rigidity} and of \eqref{dist} is that there exists a a sequence $(\mathbf R_j)_{j\in\mathbb N}\subset SO(d)$ and a constant $c>0$ (only depending on $W$ and $\Omega$) such that
\begin{equation}\label{userigidity}
c\int_\Omega|\nabla u_j|^2=\frac c{\eps_j^2}\int_\Omega|\mathbf I-\mathbf R_j+\eps_j\nabla v_j|^2\le \frac1{\eps_j^2}\int_\Omega W(\mathbf I+\eps_j\nabla v_j),
\end{equation}
and then we deduce, since  $\mathcal G_{\eps_j}(v_j)\le M$ and since \eqref{compatibility} holds, that
\begin{equation}\label{erre}\begin{aligned}
c\int_\Omega|\nabla u_j|^2+\int_\Omega H(\nabla^2 u_j)&\le M+\mathcal L(v_j)\\&= M+\frac1{\eps_j}\,\mathcal L(\mathbf R_jx-x)+\mathcal L(\mathbf R_j u_j)\le M+\mathcal L(\mathbf R_ju_j).\end{aligned}
\end{equation}
By including \eqref{kornpoincare} with $r=2\wedge p$ and H\"older inequality we obtain
\[
c\int_\Omega|\nabla u_j|^2+\int_\Omega H(\nabla^2 u_j)\le M+Q \|\nabla u_j\|_{L^2(\Omega;\R^{d\times d})}+Q\|\nabla^2 u_j\|_{L^p(\Omega;\R^{d\times d\times d})}.
\]
where $Q>0$ is a suitable constant, only depending on $\Omega,p$ and on $C_\mathcal L$ from \eqref{linearload}, and then by Young inequality we get  
\[
c\int_\Omega|\nabla u_j|^2+\int_\Omega H(\nabla^2 u_j)\le M+\frac {Q^2}{2\delta^2} +\frac{\delta^2}2\int_\Omega|\nabla u_j|^2+\frac{p-1}p\left(\frac{Q}{\delta}\right)^{\frac p{p-1}}+\frac{\delta^p}p\int_\Omega|\nabla^2 u_j|^p
\]
for every $\delta>0$. Choosing small enough $\delta$ we see that the sequence $(\nabla u_j)_{j\in\mathbb N}$ is bounded in $L^2(\Omega;\R^{d\times d})$ and that the sequence $(\nabla^2 u_j)_{j\in\mathbb N}$ is bounded in $L^p(\Omega;\R^{d\times d\times d})$.
By Poincar\'e inequality, there exists a sequence $(a_j)_{j\in\mathbb N}\subset \R^d$ and a positive constant $\mathfrak c$ (only depending on $\Omega$) such that $\|u_j-a_j\|_{L^2(\Omega;\R^d)}\le\mathfrak c \|\nabla u_j\|_{L^2(\Omega;\R^{d\times d})}$ for every $j\in\mathbb N$. Therefore, by passing to a suitable not relabeled subsequence, we have the existence of  $u\in W^{1,2}(\Omega;\R^d)$ such that $u_j-a_j\rightharpoonup u$ weakly in $L^2(\Omega;\R^d)$ and $\nabla u_j\to\nabla u$ weakly in $L^2(\Omega;\R^{d\times d})$
 as $j\to+\infty$;  the uniform bound on the sequence $(\nabla^2 u_j)_{j\in\mathbb N}$  in $L^p(\Omega;\R^{d\times d\times d})$
and Poincar\'e inequality again allow to conclude that $u\in W^{2,p}(\Omega;\R^d)$ and that $\nabla u_j\rightharpoonup \nabla u$ weakly in $W^{1,p}(\Omega;\R^{d\times d})$.

On the other hand, by \eqref{erre} and \eqref{compatibility} we have
\[
0\le -\frac1{\eps_j}\mathcal L(\mathbf R_jx-x)\le M+\mathcal L(\mathbf R_ju_j).
\] 
But then \eqref{linearload} and 
the already established uniform bounds of the sequence $(\nabla u_j)_{j\in\mathbb N}$ in $L^2(\Omega;\R^{d\times d})$ and of the sequence $(\nabla^2 u_j)_{j\in\mathbb N}$  in $L^p(\Omega;\R^{d\times d\times d})$ yield 
\[
\lim_{j\to+\infty}\mathcal L(\mathbf R_jx-x)=0.
\]
This shows that any limit point of the sequence $(\mathbf R_j)_{j\in\mathbb N}$ belongs to $\mathcal S ^0_{\mathcal L}$.
 \end{proof}

  \begin{lemma}\label{liminf2}  \AAA Let $p\ge dq/(q+1)$. \BBB Let $W$ satisfy \eqref{frame} and \eqref{coerc}.  Let $\mathcal L$ be a bounded linear functional over $W^{2,2\wedge p}(\Omega;\R^d)$ that satisfies \eqref{equi} and \eqref{compatibility}. Let $(\eps_j)_{j\in\mathbb N}\subset(0,1)$ be a vanishing sequence. Let $u\in W^{2,p}(\Omega;\R^d)$. Let $(v_j)_{j\in\mathbb N}\subset W^{2,p}(\Omega;\R^d)$ and $(\mathbf R_j)_{j\in\mathbb N}\subset SO(d)$ be sequences such that
  \[
 \nabla u_j\rightharpoonup\nabla  u \quad \mbox{weakly in $W^{1, p}(\Omega;\R^{d\times d})$}\quad\mbox{as $j\to+\infty$}, 
 \]
where  $
 u_j:=\mathbf R_j^T v_j+\frac{1}{\eps_j}(\mathbf R_j^Tx-x)
 $, and $\mathbf R_j\to\mathbf R\in \mathcal S^0_\mathcal L$ as $j\to+\infty$. Then
 \[
 \overline{\mathcal G}(u)\le\liminf_{j\to+\infty}\mathcal G_{\eps_j}(v_j). 
 \]
   \end{lemma}
 \begin{proof}
 Let 
$
T_j:=\{x\in\Omega: \sqrt{\eps_j}|\nabla u_j(x)|<1\},
$
so that 
\[
|\Omega\setminus T_j|\le \int_{\Omega\setminus T_j}\sqrt{\eps_j}|\nabla u_j|\le\sqrt{\eps_j}\int_\Omega|\nabla u_j|,
\]
thus $|\Omega\setminus T_j|\to 0$ as $j\to+\infty$. By repeating the argument in the proof of Lemma \ref{liminf}, for every $x\in T_j$ we have $\eps_j|\nabla u_j(x)|<\sqrt{\eps_j}$ and thus \eqref{modulus} holds for every large enough $j$, yielding
\[\begin{aligned}
&\liminf_{j\to+\infty}\frac1{\eps_j^2}\int_\Omega W(\mathbf I+\eps_j\nabla v_j)=\liminf_{j\to+\infty}\frac1{\eps_j^2}\int_\Omega W(\mathbf I+\eps_j\nabla u_j)\ge\liminf_{j\to+\infty}\frac1{\eps_j^2}\int_{T_j}W(\mathbf I+\eps_j\nabla u_j)\\&\qquad\ge \liminf_{j\to+\infty}\left(\frac12\int_{T_j}\mathbb E(u_j)\,D^2W(\mathbf I)\,\mathbb E(u_j)-\int_{T_j}\omega(\eps_j|\nabla u_j|)\,|\nabla u_j|^2\right)\\&\qquad \ge \liminf_{j\to+\infty}\left(\frac12\int_{\Omega}(\bb_{T_j}\mathbb E(u_j))\,D^2W(\mathbf I)\,(\bb_{T_j}\mathbb E(u_j))-\int_{T_j}\omega(\sqrt{\eps_j})|\nabla u_j|^2\right),
\end{aligned}\]
where the first equality is due to the frame indifference of $W$.
But $\omega(\sqrt{\eps_j})\to 0$ as $j\to+\infty$ and $(\nabla u_j)_{j\in\mathbb N}$ is bounded in $L^2(\Omega;\R^{d\times d})$ because of the  embedding of $W^{1,p}(\Omega)$ in $L^2(\Omega)$ yielding $\nabla u_j\to\nabla u $  \BBB in $L^2(\Omega;\R^{d\times d})$.
  Moreover,
for every $\eta\in L^2(\Omega;\R^{d\times d}) $ we have
\[
\left|\int_\Omega\eta: \bb_{\Omega\setminus T_j}\mathbb E(u_j)\right|\le \|\nabla u_j\|_{L^2(\Omega;\R^{d\times\d})}\left(\int_{\Omega\setminus T_j}|\eta|^2\right)^\frac12
\]
where the right hand side goes to zero as $j\to+\infty$ since $|\Omega\setminus T_j|\to 0$, so that  $\bb_{\Omega\setminus T_j}\mathbb E(u_j)\rightharpoonup 0$ weakly in $L^2(\Omega;\R^{d\times d})$,
and by writing
 $\bb_{T_j}\mathbb E(u_j)=\mathbb E(u_j)-\bb_{\Omega\setminus T_j}\mathbb E(u_j)$
 we see that  $\bb_{T_j}\mathbb E(u_j)\rightharpoonup \mathbb E(u)$
  weakly in $L^2(\Omega;\R^{d\times d})$. 
We conclude that 
\begin{equation}\label{bulk}
\liminf_{j\to+\infty}\frac1{\eps_j^2}\int_\Omega W(\mathbf I+\eps_j\nabla v_j)\ge  \frac12\int_{\Omega}\mathbb E(u)\,D^2W(\mathbf I)\,\mathbb E(u),
\end{equation}
thanks to the weak $L^2(\Omega;\R^{d\times d})$ semicontinuity of the map $\mathbf F\mapsto\int_\Omega \mathbf F^T\,D^2W(\mathbf I)\,\mathbf F$.

By the Sobolev embedding $W^{1,p}(\Omega;\R^{d\times d})\hookrightarrow L^2(\Omega;\R^{d\times d})$, holding since $p>d/2$, we get $\nabla u_j\rightharpoonup\nabla u$ weakly in $L^2(\Omega;\R^{d\times d})$.
By Poincar\'e inequality we deduce the existence of a sequence $(a_j)_{j\in\mathbb N}\subset\R^d$ and of $\bar u \in L^2(\Omega;\R^d)$ such that $u_j-a_j\rightharpoonup \bar u$ weakly in $L^2(\Omega;\R^{d\times d})$ and such that $\nabla \bar u=\nabla u$ on $\Omega$. Since   $\nabla^2 u_j\rightharpoonup \nabla ^2 u$ {weakly in $L^p(\Omega;\R^{d\times d\times d})$}, we deduce that  $u_j-a_j\rightharpoonup \bar u$ weakly in $W^{2,2\wedge p}(\Omega;\R^d)$. Therefore since \eqref{equi} holds and since $\mathcal L$ is a bounded linear functional over $W^{2,2\wedge p}(\Omega;\R^d)$, we get 
$
\mathcal L(u_j)=\mathcal L(u_j-a_j)\to\mathcal L(\bar u)=\mathcal L(u)
$ as $j\to+\infty$,
 and since $\mathbf R_j\to \mathbf R$, we obtain 
 \[
 \lim_{j\to+\infty} \mathcal L(\mathbf R_j u_j)=\lim_{j\to+\infty}\mathcal L(\mathbf R_j(u_j-a_j))=\mathcal L(\mathbf R \bar u)=\mathcal L(\mathbf R u).
 \]
 By taking \eqref{compatibility} into account we have
 \[
 -\mathcal L(v_j)=-\frac1{\eps_j}\mathcal L(\mathbf R_jx-x)-\mathcal L(\mathbf R_j u_j)\ge -\mathcal L(\mathbf R_j u_j)
 \]
 and therefore
 \begin{equation}\label{force}
 \liminf_{j\to+\infty} -\mathcal L(v_j)\ge -\mathcal L(\mathbf R u)\ge -\max_{\mathbf R\in\mathcal S^0_\mathcal L}\mathcal L(\mathbf R u).
 \end{equation}
 
Since $x+\eps_jv_j=\mathbf R_j(x+\eps_j u_j)$, and since $\cof(\mathbf R\mathbf F)=\mathbf R\,\cof\mathbf F$ for every $\mathbf R\in SO(d)$ and every $\mathbf F\in \R^{d\times d}$, we have
\[
\int_{\partial\Omega}||\cof(\mathbf I+\eps_j\nabla v_j)\,\mathbf n|-1|^q\,dS=\int_{\partial\Omega}||\cof(\mathbf I+\eps_j\nabla u_j)\,\mathbf n|-1|^q\,dS,
\] 
therefore Lemma \ref{keylemma} implies
\begin{equation}\label{superficie}
\AAA \liminf_{j\to+\infty}\frac1{\eps_j^q}\int_{\partial\Omega}||\cof(\mathbf I+\eps_j\nabla v_j)\,\mathbf n|-1|^q\,dS\ge \int_{\partial\Omega}|\mathbb A(u)\,\mathbf n\cdot\n|^q.\BBB
\end{equation}

The weak lower semicontinuity of the $L^p$ norm for the second gradient term, along with \eqref{bulk}, \eqref{force} and \eqref{superficie}, entails the desired result. 
\end{proof}
 
 \begin{lemma}\label{limsup2}  \AAA Let $p\ge dq/(q+1)$. \BBB Let $W$ satisfy \eqref{frame} and \eqref{coerc}.  Let $\mathcal L$ be a bounded linear functional over $W^{2,2\wedge p}(\Omega;\R^d)$ that satisfies \eqref{equi} and \eqref{compatibility}. Let $(\eps_j)_{j\in\mathbb N}\subset(0,1)$ be a vanishing sequence. For every $u\in W^{2,p}(\Omega;\R^d)$ there exist a sequence $(u_j)_{j\in\mathbb N}\subset W^{2,p}(\Omega;\R^d)$ and $\mathbf R_u\in\mathcal S^0_\mathcal L$ such that 
  \begin{equation}\label{twoconv}
  u_j\to  u \; \mbox{ strongly in $W^{2, p}(\Omega;\R^{d})$}\quad
 \end{equation}
and such that
\[
\limsup_{j\to+\infty}\mathcal G_{\eps_j}(v_j)\le\overline{\mathcal G}(u),
\]
where $v_j:=\mathbf R_u u_j+\dfrac1{\eps_j}(\mathbf R_ux-x)$.
  \end{lemma}
  \begin{proof} Let $(u_j)_{j\in\mathbb N}\subset C^\infty(\overline\Omega;\R^d)\cap W^{2,p}(\Omega;\R^d)$ be a sequence that strongly converges to $u$ in $W^{2,p}(\Omega;\R^d)$. If $p<2$, we also have $\nabla u_j\to\nabla u$ strongly in $L^2(\Omega;\R^{d\times d})$ as $j\to+\infty$ by Sobolev embedding, since \AAA  $p\ge dq/(q+1)$. \BBB By the same argument of the proof of Lemma \ref{limsup}, it is possible to assume that $\eps_j\|\nabla u_j\|_{C^0(\overline\Omega;\R^{d\times d})}\to 0$ as $j\to+\infty$, therefore by repeating the arguments therein we get 
  \[
  \lim_{j\to+\infty}\frac1{\eps_j^2}\int_\Omega W(\mathbf I+\eps_j \nabla u_j)=\frac12\int_\Omega\mathbb E(u)\,D^2W(\mathbf I)\,\mathbb E(u).
  \] 
We have $\int_\Omega|\nabla^2u_j|^p\to \int_\Omega|\nabla^2u|^p$ as $j\to+\infty$, and moreover by Lemma \ref{keylemma} we have
\[
\lim_{j\to+\infty}\frac1{\eps_j^q}\int_{\partial\Omega}||\cof(\mathbf I+\eps_j\nabla u_j)\,\n|-1|^q\,dS=\int_{\partial\Omega}|\mathbb A(u)\,\n\cdot\n|^q\,dS.
\]  
  
Let now $\mathbf R_u$ be a minimizer of the function
\begin{equation*}\mathbf R\mapsto \frac12\int_\Omega\mathbb E(u)\,D^2W(\mathbf I)\,\mathbb E(u)+\int_\Omega H(\nabla^2 u)+\gamma\int_{\partial\Omega}|\mathbb A(u)\,\n\cdot\n|^q\,dS -\mathcal L(\mathbf R u)\end{equation*}
over $\mathcal S^0_\mathcal L$,
and let $v_j:=\mathbf R_uu_j+\frac1{\eps_j}(\mathbf R_ux-x)$, $j\in\mathbb N$. We notice that by frame indifference all the terms in $\mathcal G_{\eps_j}$, excluding the load term, are the same if evaluated at $u_j$ or $v_j$. In particular, we get
\[\begin{aligned}
&\limsup_{j\to+\infty}|\mathcal G_{\eps_j}(v_j)-\overline{\mathcal G}(u)|=\limsup_{j\to+\infty}|\frac{1}{\eps_j}\mathcal L(\mathbf R_ux-x)+\mathcal L(\mathbf R_uu_j)-\mathcal L(\mathbf R_uu)|\\&\qquad=\limsup_{j\to+\infty}|\mathcal L(\mathbf R_uu_j)-\mathcal L(\mathbf R_uu)|\\&\qquad\le \limsup_{j\to+\infty}  C_{\mathcal L}\left((1+\mathfrak c)\mathfrak K \, \|\nabla u_j-\nabla u\|_{L^{2\wedge p}(\Omega;\mathbb R^{d\times d})}+\|\nabla^2 u_j-\nabla^2 u\|_{L^{2\wedge p}(\Omega;\R^{d\times d\times d})}\right)=0,
\end{aligned}\]
having used $\mathbf R_u\in\mathcal S^0_\mathcal L$,  \eqref{kornpoincare} with $r=2\wedge p$, and \eqref{twoconv}. The proof is concluded.
\end{proof}

   \begin{proofad2}
    Let $(v_j)_{j\in\mathbb N}\subset W^{2,p}(\om,\mathbb R^d)$ be a sequence of minimizers of $\mathcal G_{\eps_j}$ over $W^{2,p}(\Omega;\R^d)$. Since $\mathcal G_{\eps_j}(0)=0$, we may assume wlog that $\mathcal G_{\eps_j}(v_j)\le 1$ for every $j\in\mathbb N$. By Lemma \ref{compactness2} and Lemma \ref{liminf2}, there exist $u_*\in W^{2,p}(\Omega;\R^d)$, $\mathbf R\in\mathcal S^0_\mathcal L$ and a sequence $(\mathbf R_j)_{j\in\mathbb N}\subset SO(d)$ such that, by possibly passing to a not relabeled subsequence, there  hold $\mathbf R_j\to\mathbf R$ and $\nabla u_j\to\nabla u_*$ weakly in $W^{1,p}(\Omega;\R^{d\times d})$ as $j\to+\infty$, where $
 u_j:=\mathbf R_j^T v_j+\frac{1}{\eps_j}(\mathbf R_j^Tx-x),
 $ 
 and
 \begin{equation}\label{down}
 \overline{\mathcal G}(u_*)\le\liminf_{j\to+\infty}\mathcal G_{\eps_j}(v_j).
 \end{equation}
  Let now $ \tilde u\in W^{2,p}(\Omega;\R^d)$. By Lemma \ref{limsup2},   there exist $\mathbf R_{\tilde u}\in\mathcal S^0_\mathcal L$ and a sequence $(\tilde u_j)_{j\in\mathbb N}\subset W^{2,p}(\Omega;\R^d)$ such that, letting $\tilde v_j:=\mathbf R_{\tilde u} \tilde u_j+\frac1{\eps_j}(\mathbf R_{\tilde u}x-x)$, there holds 
  \begin{equation}\label{up}
  \limsup_{j\to+\infty}\mathcal G_{\eps_j}(\tilde v_j)\le \overline{\mathcal G}(\tilde u).
  \end{equation}
   By combining \eqref{down} and \eqref{up}, since $(v_j)_{j\in\mathbb N}$ is a sequence of minimizers, we deduce
   \begin{equation}\label{upanddown}
    \overline{\mathcal G}(u_*)\le\liminf_{j\to+\infty}\mathcal G_{\eps_j}(v_j)\le  \limsup_{j\to+\infty}\mathcal G_{\eps_j}(\tilde v_j)\le \overline{\mathcal G}(\tilde u).
   \end{equation}
   Then, the  arbitrariness of $\tilde u$ shows that $u_*$ minimizes $\overline{\mathcal G}$ over $W^{2,p}(\Omega;\R^d)$, and choosing $\tilde u=u_*$ in \eqref{upanddown} we obtain \eqref{last}.
   \end{proofad2}
 
 \UUU
 \begin{proofad3} 
 The proof is the same as that of Theorem \ref{mainth1} and Theorem \ref{main2}. The only difference is indeed in the interface terms. However, the limiting behavior of the interface terms of functional $\mathcal F_\eps$ and $\mathcal I_\eps$ is given by Corollary \eqref{corolast}, which can be used in place of Lemma \ref{keylemma}. \AAA This shows the validity of the result for the Dirichlet problem. Concerning the pure traction problem, the proof is again the same for functional $\mathcal G_\eps$, taking also advantage of the frame indifference of the interfacial term therein the allows to perform the argument in the proof of Lemma \ref{limsup2}. \BBB
 \end{proofad3}\KKK
 
 \AAA 
 In order to prove Theorem \ref{main3}, we give a preliminary compactness result.

  \begin{lemma}\label{compactness3}  Let $p\ge{dq}/{(q+1)}$. If $q=1$ and $d=2$, let also $p>1$. Let $M>0$. Let $W$ satisfy \eqref{frame} and \eqref{dist}. Let $\mathcal L$ be a bounded linear functional over $W^{2,2\wedge p}(\Omega;\R^d)$ that satisfies \eqref{equi} and \eqref{compatibility}. Let $(\eps_j)_{j\in\mathbb N}\subset(0,1)$ be a vanishing sequence. Let $(v_j)_{j\in\mathbb N}\subset W^{2,p}(\Omega;\R^d)$ be a sequence such that $\mathcal I_{\eps_j}(v_j)\le M$ for every $j\in\mathbb N$. Then there exist 
   $v\in W^{2,p}(\Omega;\R^d)$ such that  there hold 
 $
\nabla v_j\rightharpoonup\nabla  v \;\; \mbox{ weakly in $W^{1, p}(\Omega;\R^{d\times d})$}
 $ in the limit as $j\to+\infty$ (possibly along a not relabeled subsequence).
  \end{lemma}
 \begin{proof}  By \eqref{rigidity} and  \eqref{dist} there exists a sequence $(\mathbf R_j)_{j\in\mathbb N}\subset SO(d)$ such that \eqref{userigidity} holds, where $u_j$ is defined by \eqref{UJ}. Therefore
 \[
 \begin{aligned}
& c\int_\Omega|\nabla u_j|^2+\int_\Omega H(\nabla^2 u_j)+\frac\gamma{\eps_j^q}\int_{\partial\Omega}|(\cof(\mathbf I+\eps_j\nabla v_j)-\mathbf I)\,\n|^q\,dS\\
&\qquad\le\frac1{\eps_j^2}\int_\Omega W(\mathbf I+\eps_j\nabla v_j)+\frac\gamma{\eps_j^p}\int_\Omega H(\eps_j\nabla^2 v_j)+\frac1{\eps_j^q}\int_{\partial\Omega}|(\cof(\mathbf I+\eps_j\nabla v_j)-\mathbf I)\,\n|^q\,dS\\
&\qquad\le \mathcal I_{\eps_j}(v_j)+\mathcal L(v_j)\le M+\frac1{\eps_j}\mathcal L(\mathbf R_jx-x)+\mathcal L(\mathbf R_ju_j)\le  M+\mathcal L(\mathbf R_ju_j)
 \end{aligned}
 \]
having used \eqref{compatibility}. 
As seen in the proof of Lemma \ref{compactness2}, this shows that there exists $u\in W^{2,p}(\Omega;\R^d)$ such that $\nabla u_j\rightharpoonup \nabla u$ weakly in $W^{1,p}(\Omega;\R^{d\times d})$ along a suitable not relabeled subsequence, and that
$$\frac1{\eps_j^q}\int_{\partial\Omega}|(\cof(\mathbf I+\eps_j\nabla v_j)-\mathbf I)\,\n|^q\,dS=\int_{\partial\Omega}\left|\frac{\mathbf R_j(\cof(\mathbf I+\eps_j\nabla u_j)-\mathbf I)\,\n}{\eps_j}+\frac{(\mathbf R_j-\mathbf I)\,\n}{\eps_j}\right|^q\,dS$$
is uniformly bounded w.r.t. $j$. However the sequence $(\eps_j^{-1}\mathbf R_j(\cof(\mathbf I+\eps_j\nabla u_j)-\mathbf I)\,\n)_{j\in\mathbb N}$ is uniformly bounded in $L^q(\partial\Omega;\mathbb R^{d})$ thanks to Corollary \ref{coro}. Therefore we deduce that
\begin{equation}\label{zn}
\sup_{j\in\mathbb N}\int_{\partial\Omega}\left|\frac{(\mathbf R_j-\mathbf I)}{\eps_j}\,\n\right|^q\,dS<+\infty.
\end{equation}
Let $\mathbf V_j:=\eps_j^{-1}(\mathbf R_j-\mathbf I)$ and $\mathbf Z_j:=\mathbf V_j/|\mathbf V_j|$. We claim that the sequence $(\mathbf V_j)_{j\in\mathbb N}\subset \mathbb R^{d\times d}$ is bounded. Indeed, suppose not. Then there exists a suitable subsequence along which $|\mathbf V_j|$ diverge  and $\mathbf Z_j\to\mathbf Z$ for some suitable $\mathbf Z\in \mathbb R^{d\times d}$ with $|\mathbf Z|=1$.
Thus 
\[
|\mathbf V_j|^{-q}\int_{\partial\Omega}|\mathbf V_j\,\n|^q\,dS=\int_{\partial\Omega}|\mathbf Z_j\,\n|^q\,dS\to\int_{\partial\Omega}|\mathbf Z\,\n|^q\,dS
\]
as $j\to+\infty$.
But $|\mathbf V_j|\to+\infty$ and \eqref{zn} then imply $ \int_{\partial\Omega}|\mathbf Z\,\n|^q\,dS=0$, which is a contradiction since $|\mathbf Z|=1$.
The claim is proven and it implies the result, since $\nabla u_j\rightharpoonup \nabla u$ weakly in $W^{1,p}(\Omega;\R^{d\times d})$ and since $u_j$ and $v_j$ are related by \eqref{UJ}.
 \end{proof}

\begin{proofad4} 
We preliminarily  notice that existence of minimizers over $W^{2,p}(\Omega;\R^d)$ of $\mathcal I_{\eps_j}$, for every fixed $j$, and of $\mathcal I_*$ are obtained in the same way as done in Lemma \ref{nonlineare} and Lemma \ref{barg}, respectively.

We first check Gamma liminf inequality, that is  we let $v\in W^{2,p}(\Omega;\R^d)$, we  let $(v_j)_{j\in\mathbb N}\subset W^{2,p}(\Omega;\R^d)$ be a  sequence such that
  $
 \nabla v_j\rightharpoonup\nabla  v$ weakly in $W^{1, p}(\Omega;\R^{d\times d})$ {as $j\to+\infty$}, 
and we check that
 \[
 {\mathcal I}_*(v)\le\liminf_{j\to+\infty}\mathcal I_{\eps_j}(v_j). 
 \]
 This is obtained in the very same way as in the proof of Lemma \ref{liminf2}. Indeed, after defining $
T_j:=\{x\in\Omega: \sqrt{\eps_j}|\nabla v_j(x)|<1\}
$
we follow the argument therein and obtain \eqref{bulk} with $v$ in place of $u$. Similarly by using Poincar\'e inequality and \eqref{equi} we get $\mathcal L(v_j)\to \mathcal L(v)$
 as $j\to+\infty$. We also have lower semicontinuity of $L^p$ norm of the second gradient, as well as lower semicontinuity of the interfacial term, thanks to Corollary \ref{coro}.

As second step we check Gamma limsup inequality, that is we check that for every $v\in W^{2,p}(\Omega;\R^d)$ there exists a sequence $(v_j)_{j\in\mathbb N}\subset W^{2,p}(\Omega;\R^d)$  such that 
  $  v_j\to  v \; \mbox{ strongly in $W^{2, p}(\Omega;\R^{d})$}$
and such that
\[
\limsup_{j\to+\infty}\mathcal I_{\eps_j}(v_j)\le{\mathcal I}_*(v).
\]
 The argument is the same that was used for proving Lemma \ref{limsup2}. Let $(v_j)_{j\in\mathbb N}\subset C^\infty(\overline\Omega;\R^d)\cap W^{2,p}(\Omega;\R^d)$ be a sequence that strongly converges to $v$ in $W^{2,p}(\Omega;\R^d)$, constructed as in proof of Lemma \ref{limsup}. Thus we get continuity along this sequence of all the terms in the energy but the load term, also using Corollary \ref{coro}.
 Concerning the load term, we directly get $\mathcal L(v_j)\to\mathcal L(v)$ by means of \eqref{kornpoincare}.
 
 Having proven compactness and Gamma convergence, the result follows.
 \end{proofad4}

 \BBB

\section*{Acknowledgements}  \noindent
\AAB We are grateful to G.~Francfort for bringing reference \cite{Dias} to our attention and for providing valuable feedback on the earlier version of this manuscript. \BBB The work of MK was consecutively  supported  by the GA\v{C}R projects 21-06569K and 23-04766S Partial support from the M\v{S}MT \v{C}R project ERC CZ  No. LL2310 is also acknowledged. He is indebted to  Dipartimento di Ingegneria Meccanica, Energetica, Gestionale e dei Trasporti of the University of Genoa for hospitality and support during his stay there. EM is supported by the MIUR-PRIN project 202244A7YL.
He acknowledges hospitality of the Institute of Information Theory and Automation of the
	Czech Academy of Sciences where part of this research was developed.
\begin{flushleft}

\end{flushleft}

\end{document}